\documentclass [11pt]{amsart}
\usepackage {amsmath, amssymb, amscd, mathrsfs, stmaryrd, hyperref,  bbm, enumerate, url, graphicx, color}
\usepackage[all, cmtip]{xy}
\usepackage[text={6.5in,9in},centering,letterpaper,dvips]{geometry}

\newtheorem {theorem}{Theorem}[section]
\newtheorem {lemma}[theorem]{Lemma}
\newtheorem {proposition}[theorem]{Proposition}
\newtheorem {corollary}[theorem]{Corollary}

\newtheorem {definition}[theorem]{Definition}
\newtheorem {question}[theorem]{Question}

\theoremstyle{remark}
\newtheorem {remark}[theorem]{Remark}
\newtheorem {example}[theorem]{Example}

\def\Z {{\mathbb{Z}}}
\def\N {{\mathbb{N}}}
\def\R {{\mathbb{R}}}

\def\Q {{\mathbb{Q}}}
\def\CP{\mathbb{CP}}

\def\Kh{\operatorname{Kh}}
\def\KhRN{\operatorname{KhR}_N}
\def\bKhRN{\textbf{KhR}_N}
\def\bKhR{\textbf{KhR}}
\def\cKhR{\underline{\operatorname{KhR}}}
\def\cKhRNa{\underline{\operatorname{KhR}}_{N, \alpha}}
\def\KhR{\operatorname{KhR}}

\def\S{\mathcal{S}^N}
\def\St{\mathcal{S}^2}
\def\Sz{\mathcal{S}_0^N}

\def\bL{\overline{L}}
\def\gl{\mathfrak{gl}}
\def\sl{\mathfrak{sl}}
\def\del{\partial}

\def\k{\mathbbm{k}}
\def\hZ{\widehat{Z}}
\def\Aut{\operatorname{Aut}}
\def\Hom{\operatorname{Hom}}
\def\Ext{\operatorname{Ext}}
\def\Sym{\operatorname{Sym}}
\def\A{\mathcal{A}}

\def\Cr{\mathfrak{C}}
\def\F{\mathcal{F}}
\def\HH{\operatorname{HH}}
\newcommand{\Hn}[2]{{}_{#1}(H^n)_{#2}}
\newcommand{\Hnp}[2]{{}_{#1}(H^{n+1})_{#2}}
\newcommand{\M}[2]{{}_{#1}M_{#2}}
\newcommand{\m}[2]{{}_{#1}m_{#2}}
\def\mm{\mathbf{m}}
\def\tangle{J}
\def\Id{\operatorname{Id}}
\def\Out{\operatorname{Out}}
\def\ta{\tilde{a}}

\begin{document}

\title{Skein lasagna modules for $2$-handlebodies}

\author[Ciprian Manolescu]{Ciprian Manolescu}
\thanks {The authors were supported by NSF grants DMS-1708320 and DMS-2003488.}
\address {Department of Mathematics, Stanford University\\ 
Stanford, CA 94305, USA}
\email {cm5@stanford.edu}

\author[Ikshu Neithalath]{Ikshu Neithalath}
\address{Centre for Quantum Mathematics, University of Southern Denmark\\
5430 Odense M, Denmark}
\email{ikshu@imada.sdu.dk}

\begin{abstract}
Morrison, Walker, and Wedrich used the blob complex to construct a generalization of Khovanov-Rozansky homology to links in the boundary of a $4$-manifold. The degree zero part of their theory, called the skein lasagna module, admits an elementary definition in terms of certain diagrams in the $4$-manifold. We give a description of the skein lasagna module for $4$-manifolds without $1$- and $3$-handles, and present some explicit calculations for disk bundles over $S^2$.
\end{abstract}

\subjclass[2020]{57K18 (Primary); 57R56, 57K41 (Secondary)}
\maketitle

\section{Introduction}
Over the past twenty years, categorified knot invariants have been a central topic in low dimensional topology. The starting point was Khovanov's categorification of the Jones polynomial \cite{Kh}. This was generalized by Khovanov and Rozansky in \cite{KhR} to a sequence of link homology theories $\KhRN$ for $N \geq 1$, where Khovanov homology corresponds to $N=2$. Khovanov homology has been successfully used to give new, combinatorial proofs of deep results about smooth surfaces in 4-manifolds, such as the Milnor Conjecture \cite{Rasmussen} and the Thom Conjecture \cite{PLC}, for which the original proofs involved gauge theory \cite{KMMilnor, KMThom}. Furthermore, by now Khovanov homology has found its own novel topological applications, as for example in the work of Piccirillo \cite{Pic1, Pic2}. Still, compared to the invariants derived from gauge theory or Heegaard Floer homology, Khovanov homology has its limitations, due to the fact that its construction is {\em a priori} just for links in $\R^3$. In particular, a major open question is whether Khovanov or Khovanov-Rozansky homology can say something new about the classification of smooth $4$-manifolds.

In \cite{MWW}, Morrison, Walker, and Wedrich proposed an extension of Khovanov-Rozansky  homology to links in the boundaries of arbitrary oriented $4$-manifolds. Specifically, they define an invariant 
$$\S \!(W; L)=\bigoplus_{b\in \Z} \S_b(W; L) =\bigoplus_{b \in \Z} \Bigl( \bigoplus_{i, j\in \Z} \S_{b, i, j}(W; L) \Bigr),$$ which is a triply-graded Abelian group associated to a smooth, oriented $4$-manifold $W$ and a framed link $L \subset \del W$. Two of the gradings ($i$ and $j$) are the usual ones in Khovanov-Rozansky homology, and the third, the {\em blob grading} $b$, is new. The construction of $\S(W; L)$ starts by defining the part in blob degree zero, $\Sz(W; L)$, in a manner reminiscent to that of the skein modules of $3$-manifolds. The group $\Sz(W; L)$, which we call the {\em skein lasagna module}, is generated by certain objects called {\em lasagna fillings} of $W$ with boundary $L$, modulo an equivalence relation that captures the ``local'' cobordism relations in Khovanov-Rozansky homology. Once $\Sz(W; L)$ is defined, the higher degree groups $\S_b(W; L)$ for $b > 0$ are obtained from it using the machinery of blob homology from higher category theory \cite{MWblob}.

It is shown in \cite{MWW} that, when $W=B^4$, the invariant $\Sz(W; L)$ recovers the Khovanov-Rozansky homology $\KhRN(L)$, and $\S_b(W; L)=0$ for $b > 0$. While computing blob homology in general is rather daunting, the skein lasagna module $\Sz(W; L)$ has a relatively simple definition.  Our goal here is to describe $\Sz(W; L)$ for a large class of non-trivial $4$-manifolds $W$ and links $L \subset \del W$.

Precisely, we will be concerned with {\em $2$-handlebodies}, that is, four-dimensional manifolds $W$  obtained from $B^4$ by attaching $n$ $2$-handles. A Kirby diagram for such a manifold consists of a framed, $n$-component link $K \subset S^3$. For every homology class $\alpha \in H_2(W;\Z) \cong \Z^n$, we define the {\em cabled Khovanov-Rozansky homology} $\cKhRNa(K)$ as the direct sum of the Khovanov-Rozansky homologies of an infinite collection of cables of $K$, modulo a certain equivalence relation. (The exact definition is given in Section~\ref{sec:cabled}.)

The skein lasagna module naturally decomposes according to the relative homology classes of lasagna fillings:
$$\Sz(W;L)=\bigoplus\limits_{\alpha\in H_2(W,L;\Z)} \Sz(W;L,\alpha).$$

Our main result is the following.

\begin{theorem}\label{2handles}
Let $W$ be the 4-manifold obtained from attaching $2$-handles to $B^4$ along a framed $n$-component link $K$.  For each $\alpha \in H_2(W;\Z) \cong \Z^n$, we have an isomorphism
\begin{align*}
\Phi: \cKhRNa(K) \xrightarrow{\phantom{a}\cong\phantom{a}} \Sz(W;\emptyset,\alpha).
\end{align*}
\end{theorem}

In general, if we want to apply Theorem~\ref{2handles} to specific examples, we run into the difficulty of  calculating Khovanov-Rozansky homology for an infinite family of cables. Nevertheless, we can do an explicit calculation when $K$ is the $0$-framed unknot, so that $W=S^2 \times D^2$.

\begin{theorem}
\label{thm:S2D2}
The skein lasagna module $\Sz(S^2 \times D^2; \emptyset)$ is supported in homological degree $0$ and has the structure of a commutative ring. We have a ring isomorphism
$$\S_{0,0,*}(S^2 \times D^2; \emptyset)\cong \Z[A_1,\dots, A_{N-1}, A_0, A_0^{-1}]$$
where the $A_k$ have quantum degree $-2k$. Under this isomorphism, the subgroup of lasagna fillings with relative homology class $\alpha\in\Z$ is identified with the subgroup of homogeneous polynomials of degree $\alpha$.
\end{theorem}

In particular, for $N=2$ (which corresponds to Khovanov homology), we have
$$\St_{0,0,j} (S^2 \times D^2; \emptyset, \alpha) \cong \begin{cases}
\Z & \text{if } j = -2k, \ k \geq 0 \\
0  & \text{otherwise.}
\end{cases}$$

When $N =2$, we also get some partial information for the $p$-framed unknot for $p \neq 0$. Then, $W$ is the $D^2$-bundle over $S^2$ with Euler number $p$, which we denote by $D(p)$.

\begin{theorem}
\label{thm:Dp}
For $p > 0$ and $N=2$, the part of the skein lasagna module of $D(p)$ that lies in class $\alpha = 0$ and homological degree $0$ is 
$$\St_{0,0,*} (D(p); \emptyset, 0) = 0.$$
On the other hand, for $p < 0$ we have
$$\St_{0,0,j} (D(p); \emptyset, 0)=\begin{cases} \Z & \text{if } \ j=0,\\
0 & \text{otherwise.}\end{cases}$$
\end{theorem}

While Theorem~\ref{2handles} was formulated for the case where the link $L \subset \del W$ is empty, we can also handle the case of ``local'' links in $\del W$, that is, those contained in a ball $B^3 \subset \del W$. Indeed, we have the following tensor product formula for boundary connected sums. To state it, it is convenient to work with coefficients in a field $\k$, and write $\S(W; L;\k)$ for the corresponding skein lasagna module.

\begin{theorem}\label{connectsum}
Let $W_1$ and $W_2$ be 4-manifolds with framed links $L_i\subset \partial W_i$ and let $W_1\natural W_2$ denote their boundary connected sum along specified copies of $B^3\subset\partial W_i$ away from the links $L_i$. Then, 
\begin{align*}
\Sz(W_1\natural W_2; L_1\cup L_2; \k)\cong \Sz(W_1; L_1; \k)\otimes \Sz(W_2; L_2; \k)
\end{align*}
\end{theorem}

Applying Theorem~\ref{connectsum} to $(W_1;L_1)=(W;\emptyset)$ and $(W_2;L_2)=(B^4;L)$, we obtain:

\begin{corollary}
Let $W$ be a 4-manifold and $L\subset B^3\subset\partial W$ a framed link contained within a ball in the boundary of $W$. Then we have
\begin{align*}
\Sz(W;L; \k)\cong \Sz(W;\emptyset; \k)\otimes \KhRN(L; \k).
\end{align*}
\end{corollary}

Furthermore, we can study skein lasagna modules associated to closed $4$-manifolds. If the boundary  of a $2$-handlebody $W$ is $S^3$, by attaching a ball we obtain a simply connected smooth $4$-manifold $X$. We can associate to $X$ the skein lasagna module $\Sz(X; \emptyset)$. 
\begin{proposition}
\label{prop:4h}
If $X$ a closed smooth $4$-manifold, then $$\Sz(X; \emptyset)\cong \Sz(X\setminus B^4;\emptyset).$$ 
\end{proposition}
In particular, we see that $\Sz(S^4; \emptyset)\cong\Z$.

It is an open question \cite[Problem 4.18]{Kirby} whether every closed, smooth, simply connected four-manifold $X$ admits a perfect Morse function or, equivalently, a Kirby diagram without 1-handles or 3-handles; i.e., whether $X\setminus B^4$ is a $2$-handlebody. In practice, many four-manifolds are known to be of this form. The list of such manifolds include:
\begin{itemize}
\item $\mathbb{CP}^2$ and $S^2 \times S^2$;
\item The elliptic surfaces $E(n)$, including the K3 surface $E(2)$; see for example \cite[Figure 8.15]{GS};
\item More generally, the log transforms $E(n)_p$; see \cite[Corollary 8.3.17]{GS};
\item The Dolgachev surface $E(1)_{2,3}$ and a few other elliptic surfaces of the form $E(n)_{p,q}$; see \cite{AkbulutDolgachev}, \cite{YasuiElliptic};
\item Smooth hypersurfaces in $\CP^3$; see for example \cite[Section 12.3]{AkbulutBook};
\item The cyclic $k$-fold branched covers $V_{k}(d) \to \CP^2$ over curves of degree $d$, with $k |d$; cf. \cite[Section 12.3]{AkbulutBook};
\item The Lefschetz fibrations $X(m,n)$ and $U(m,n)$ obtained as branched covers over curves in Hirzebruch surfaces; cf. \cite[Figures 8.31 and 8.32]{GS}.
\end{itemize}

Theorem~\ref{2handles}, combined with Proposition~\ref{prop:4h}, is a step towards understanding  the skein lasagna modules for $4$-manifolds from the above list. In particular, when $X = \mathbb{CP}^2$ or $\overline{\mathbb{CP}^2}$, we have a Kirby diagram  with a single $(\pm 1)$-framed unknot, and Theorem~\ref{thm:Dp} tells us its $N=2$ skein lasagna module in class $\alpha=0$ and homological degree zero:
\begin{equation}
\label{eq:CP2}
 \St_{0,0,*}(\mathbb{CP}^2; \emptyset, 0) =0, \ \ \ \  \St_{0,0,0}(\overline{\mathbb{CP}^2}; \emptyset, 0) \cong \Z.
 \end{equation}

For an invariant to be effective at detecting manifolds, one needs to be able to extract finite data from it. The calculations above indicate that the skein lasagna modules can have infinite rank overall, but may be finitely generated when we fix the bi-grading and the class $\alpha$.

\begin{question}
Is it true that for every $4$-manifold $W$, framed link $L \subset W$, class $\alpha \in H_2(W, L;\Z)$, and values $i, j \in \Z$, the skein lasagna module $\S_{0, i, j}(W; L, \alpha)$ is finitely generated?
\end{question}

Note that for the skein modules of closed $3$-manifolds, a finite dimensionality result was recently proved by Gunningham, Jordan and Safronov \cite{CJS}.

A more ambitious problem is the following:
\begin{question}
Can the invariant $\S_{0}(W; L)$ detect exotic smooth structures on the $4$-manifold $W$?
\end{question}

One indication that the answer might be positive is the behavior under orientation reversal. It is known that unitary TQFTs (which are symmetric under orientation reversal) cannot detect exotic smooth structures on simply connected $4$-manifolds; see \cite{FKNSWW}. On the other hand, the Donaldson and Seiberg-Witten invariants (which can detect exotic smooth structures) are highly sensitive to the orientation. The skein lasagna modules $\S_{0}(W; L)$ are constructed from the cobordism maps on Khovanov-Rozansky homology, which are also sensitive to orientation. In fact, in the case $W=\mathbb{CP}^2$ and $N=2$, one can see explicitly from \eqref{eq:CP2} that the invariants of $W$ and $\overline{W}$ are quite different.

\medskip 
\textbf{Organization of the paper.} In Section~\ref{sec:slm} we review the definition of skein lasagna modules and establish some simple properties, including Proposition~\ref{prop:4h}. In Section~\ref{sec:cabled} we give the definition of cabled Khovanov-Rozansky homology. In Section~\ref{sec:2h} we prove Theorem~\ref{2handles}. In Sections~\ref{sec:D0} and \ref{sec:Dp} we do our explicit calculations from Theorems~\ref{thm:S2D2} and \ref{thm:Dp}. In Section~\ref{sec:ConnSum} we prove the connected sum formula, Theorem~\ref{connectsum}.

\medskip 
\textbf{Acknowledgements.} We are grateful to John Baldwin, Gage Martin, Sucharit Sarkar, Paul Wedrich and Mike Willis for helpful conversations. We also thank the referees for helpful comments on the paper.

\section{Skein lasagna modules}
\label{sec:slm}
\subsection{Conventions for Khovanov-Rozansky homology}
\label{sec:framings}
In this paper we follow \cite{MWW} and, for a framed link $L \subset \R^3$, we write 
$$\KhRN(L) = \bigoplus_{i, j\in \Z} \KhRN^{i, j}(L)$$ 
for the $\gl_N$ version of Khovanov-Rozansky homology. Here, $i$ denotes the homological grading and $j$ denotes the quantum grading. For a bi-graded group $W$, we will denote by $W\{l\}$ the result of shifting its second grading (in our case, the quantum grading) by $l$:
$$W\{l\}^{i,j} = W^{i, j-l}.$$ 
For example, the invariant of the $0$-framed unknot is the commutative Frobenius algebra
\begin{equation}
\label{eq:A}
\A := \KhRN(U,0) =H^*(\CP^{N-1})\{1-N\} =( \Z[X]/\langle X^N \rangle) \{1-N\},
\end{equation}
with $1$ in bidegree $(0, 1-N)$ and multiplication by $X$ changing the bidegree by $(0, 2)$. The comultiplication on $\A$ (which corresponds to a pair-of-pants cobordism) is given by
\begin{equation}
\label{eq:comult}
 \Delta(X^m) = \sum_{k=0}^{N-m-1} X^{k+m} \otimes X^{N-1-k}
 \end{equation}
and the counit on $\A$ is
\begin{equation}
\label{eq:counit}
 \epsilon(X^m)=0 \text{ for } 0 \leq m \leq N-2\ ;  \ \ \ \  \epsilon(X^{N-1}) = 1.
\end{equation}

Note that $\KhRN$ is an invariant of {\em framed} links. We will distinguish it from the original $\sl_N$ version of Khovanov-Rozansky homology from \cite{KhR}, which we denote by $\bKhRN$ and is independent of the framing. Further, $\bKhRN$ was defined only over $\Q$ whereas $\KhRN$ has coefficients in $\Z$. The two theories differ by a shift in quantum grading:
$$ \KhRN(L) \otimes \Q \cong \bKhRN(L)\{Nw\},$$
where $w$ is the writhe of a diagram in which the given framing of $L$ is the blackboard framing. 

In the case $N=2$, we also have the ordinary Khovanov homology $\Kh(L)$ defined in \cite{Kh}. As noted in \cite{KhR}, we have
$$  \bKhR_2^{i, j}(L) \cong \Kh^{i, -j}(\bL) \otimes \Q,$$
where $\bL$ is the mirror of $L$. Moreover, from \cite[Section 7.3]{Kh}, we know that the Khovanov complexes of $L$ and $\bL$ are related by duality, and therefore the Khovanov homologies are related by the universal coefficients theorem:
\begin{equation}
\label{eq:Ext}
 \Kh^{i, j} (\bL) \cong \Hom( \Kh^{-i, -j}(L), \Z) \oplus \Ext(\Kh^{1-i, -j}(L), \Z).
\end{equation}

In this paper we will not use $\bKhRN$. We will mostly work with $\KhRN$, but in Section~\ref{sec:Dp}  we will need to relate it to $\Kh$ because the relevant calculations in the literature are done in terms of $\Kh$. The relation between the two theories is given by a (non-canonical) isomorphism:
\begin{equation}
\label{eq:KhR2}
 \KhR_2^{i, j}(L) \cong \Kh^{i, -j-2w}(\bL).
\end{equation}
 
The usual Khovanov homology $\Kh(L)$ is functorial under cobordisms in $\R^3 \times [0,1]$, but only up to sign \cite{KhCob, Jacobsson}. On the other hand, the $\gl_2$ version and, more generally, the  $\gl_N$ homology $\KhRN(L)$, are functorial over $\Z$ \cite{Blanchet, ETW}. Furthermore, it is shown in \cite{MWW} that $\KhRN(L)$ is a well-defined invariant of framed links in $S^3$, and is functorial under framed cobordisms in $S^3 \times [0,1]$. Given a framed cobordism $\Sigma \subset S^3  \times [0,1]$ from $L_0$ to $L_1$, the induced map
$$ \KhRN(\Sigma) : \KhRN(L_0) \to \KhRN(L_1)$$
is homogeneous of bidegree $(0, (1-N)\chi(\Sigma))$. For $N=2$, this map agrees with the usual cobordism map 
$$\Kh(\overline{\Sigma}): \Kh(\bL_0)\to \Kh(\bL_1),$$ up to pre- and post-composition with the isomorphisms \eqref{eq:KhR2}. 

If we have an oriented manifold $S$ diffeomorphic to the standard $3$-sphere $S^3$, and a framed link $L \subset S$, we can define a canonical invariant $\KhRN(S,L)$ as in \cite[Definition 4.12]{MWW}. When $S$ is understood from the context, we will drop it from the notation and simply write $\KhRN(L)$. 

\subsection{Definition}
Let us review the construction of skein lasagna modules following \cite[Section 5.2]{MWW}. 

Let $W$ be a smooth oriented $4$-manifold and $L \subset \del W$ a framed link. A {\em lasagna filling} $F=(\Sigma, \{(B_i,L_i,v_i)\})$ of $W$ with boundary $L$ consists of
\begin{itemize}
\item A finite collection of disjoint $4$-balls $B_i$ (called {\em input balls}) embedded in the interior or $W$;
\item A framed oriented surface $\Sigma$ properly embedded in $W \setminus \cup_i B_i$, meeting $\del W$ in $L$ and meeting each $\del B_i$ in a link $L_i$; and
\item for each $i$, a homogeneous label $v_i \in \KhRN(\del B_i, L_i).$
\end{itemize}
The bidegree of a lasagna filling $F$ is 
$$ \deg(F) := \sum_i \deg(v_i) + (0,(1-N)\chi(\Sigma)).$$
If $W$ is a $4$-ball, we can upgrade the functoriality of $\Kh$ to define a cobordism map
$$ \KhRN(\Sigma): \bigotimes_i \KhRN(\del B_i, L_i) \to \KhRN(\del W, L)$$
and an evaluation
$$ \KhRN(F) := \KhRN(\Sigma)(\otimes_i v_i) \in \Kh(\del W, L).$$

We now define the skein lasagna module to be the bi-graded Abelian group
$$ \S_0(W; L) := \Z\{ \text{lasagna fillings $F$ of $W$ with boundary $L$}\}/\sim$$
where $\sim$ is the transitive and linear closure of the following relation: 
\begin{itemize}
\item Linear combinations of lasagna fillings are set to be multilinear in the labels $v_i$;
\item Furthermore, two lasagna fillings $F_1$ and $F_2$ are set to be equivalent if  $F_1$ has an input ball $B_1$ with label $v_1$, and $F_2$ is obtained from $F_1$ by replacing $B_1$ with another lasagna filling $F_3$ of a $4$-ball such that $v_1=\KhRN(F_3)$, followed by an isotopy rel boundary:
\[ 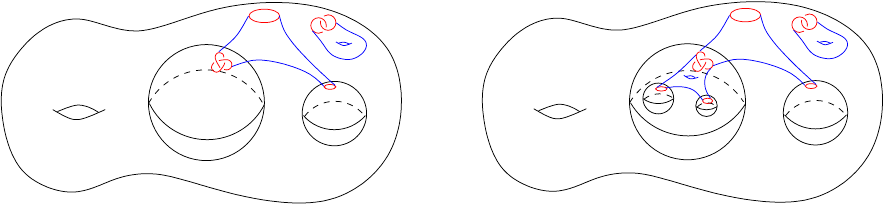
\]
\end{itemize} 

\subsection{Homology classes of lasagna fillings}
Given a lasagna filling $F$ for $L\subset \partial W$ specified by the data $(\Sigma, \{(B_i,L_i,v_i)\})$, we denote by $[F]$ its equivalence class in $\Sz(W; L)$. We also define the homology class of $F$, denoted $\llbracket F \rrbracket \in H_2(W,L;\Z)$ by
\begin{align*}
\llbracket F \rrbracket = &  \bigl[\bigl(\Sigma,L\cup ( \cup_i L_i )\bigr)]\in H_2(W,L\cup (\cup_i \partial B_i);\Z)\cong H_2(W, L;\Z)
\end{align*}
where $\bigl[\bigl(\Sigma,L\cup ( \cup_i L_i )\bigr)]$ denotes the relative fundamental class of the surface $\Sigma$. If $F\sim F'$ are equivalent fillings in $\Sz(W; L)$, then $F$ and $F'$ agree up to isotopy outside of some disjoint balls. In particular, they are homologous relative to those balls, and thus homologous in $H_2(W,L;\Z)$. Thus, an equivalence class $[F]$ of lasagna fillings has a well-defined homology class $\llbracket F \rrbracket$.

Given $\alpha\in H_2(W,L;\Z)$, let $\Sz(W;L,\alpha)$ denote the subgroup of $\Sz(W;L)$ generated by fillings with homology class $\alpha$. Since $\Sz(W;L)$ is generated by lasagna fillings and the fillings are partitioned according to their homology class, we obtain a decomposition
\begin{align*}
\Sz(W;L)=\bigoplus\limits_{\alpha\in H_2(W,L;\Z)} \Sz(W;L,\alpha).
\end{align*}

\subsection{Adding 3- and 4-handles}
Smooth  $4$-manifolds admit handle decompositions, which are represented pictorially by Kirby diagrams \cite{GS}. To compute $\Sz(W; \emptyset)$ for $W$ an arbitrary smooth compact $4$-manifold, we would need to understand how adding $k$-handles to $W$ affects $\Sz(W;\emptyset)$. We will discuss $2$-handles in detail in Section~\ref{sec:2h}, and we cannot say much about $1$-handles. For now, we present the following result about $3$- and $4$-handles.

\begin{proposition}\label{3and4handles}
Let $i: W\to W'$ be the inclusion of a 4-manifold $W$ into $W'$. Then we have a natural map $$i_*:\Sz(W; \emptyset)\to \Sz(W';\emptyset).$$ If $W'$ is the result of a $k$-handle attachment to $W$, then $i_*$ is a surjection for $k=3$, and an isomorphism for $k=4$.
\end{proposition}

\begin{proof}
Let $[F]\in\Sz(W; \emptyset)$ by the class of a lasagna filling $F$ with surface $\Sigma$. Let $i(F)$ denote the filling $F$ viewed inside of $W'$. If $F \sim \widetilde{F}$, then clearly $[i(F)]=[i(\widetilde{F})]$. Therefore, we have a well-defined map $i_*: \Sz(W; \emptyset)\to \Sz(W';\emptyset)$ given by $i_*([F])=[i(F)]$.

We observe that $i_*$ is surjective if every surface $\Sigma'\subset W'$ can be isotoped to lie in $W$. Consider the case of $W'$ being the result of attaching a $k$-handle $h$ to $W$. Removing the cocore of $h$ from $W'$ produces a manifold that deformation retracts to $W$. In particular, if $\Sigma$ can be isotoped to not intersect the cocore of $h$, then $\Sigma$ can be isotoped to lie entirely in $W$. By transversality, this occurs when
\begin{align*}
\operatorname{dim}(\Sigma) + \operatorname{dim}(\operatorname{cocore}(h))<4\\
2+(4-k)<4\\
2<k
\end{align*}

When $k=3$ or $4$, the cocore is $1$- or $0$-dimensional, and hence embedded surfaces in $W'$ can be isotoped off the handle. Thus, if $W'$ is the result of attaching a 3-handle or 4-handle to $W$, then $i_*$ is surjective. 

If $i: W\to W'$ is a 4-handle addition, then surfaces in $W'$ can be isotoped off the handle even in a one-parameter family. Therefore, if two lasagna fillings are equivalent in $W'$, after we isotope them to lie in $W$ they are still equivalent in $W$. It follows that $i_*$ is injective, and therefore an isomorphism.
\end{proof}

\begin{proof}[Proof of Proposition~\ref{prop:4h}] This is an immediate corollary of Proposition~\ref{3and4handles}, because a closed $4$-manifold $X$ is obtained from $X \setminus B^4$ by attaching a $4$-handle.\end{proof}

\section{The cabled Khovanov-Rozansky homology}
\label{sec:cabled}
\subsection{The general definition}
Let $K\subset S^3$ be a framed oriented link with components $K_1,\dots, K_n$. Fix also two $n$-tuples of nonnegative integers, $k^-=(k_1^-,\dots,k_n^-)$ and $k^+=(k_1^+,\dots,k_n^+)$. Let 
$K(k^-,k^+)$ denote the framed, oriented link obtained from $k_i^-$ negatively oriented parallel strands to $K_i$ (where the choice of parallel strand is determined by the framing) and $k_i^+$ positively oriented parallel strands. The framing on each of the parallel strands is the same as the framing on the corresponding knots $K_i$.

To be more precise, using the framing, we get a diffeomorphism $f_i$ between a tubular neighborhood of $K_i$ and $S^1 \times D^2$. For each $i$, we pick distinct points 
$$x_1^-, \dots, x_{k_i^-}^-, x_1^+, \dots, x_{k_i^+}^+ \in D^2.$$
Then 
$$K(k^-,k^+) = \bigcup_i f_i^{-1}(S^1 \times \{x_1^-, \dots, x_{k_i^-}^-, x_1^+, \dots, x_{k_i^+}^+\})$$  

\begin{remark}
Suppose $n=1$ so that $K$ is a knot with some framing coefficient $p$. Then, as an unoriented link,  $K(k^-,k^+)$ is the $(p(k^-+k^+), (k^-+k^+))$ cable of $K$.
\end{remark}

Let $B_n$ be the braid group on $n$ strands, and $F: B_n \to S_n$ the natural homomorphism to the symmetric group. For $0 \leq k \leq n$, let $B_{k, n-k} = F^{-1}(S_k \times S_{n-k})$. Thus, if we view the braid group as the mapping class group of the punctured disk, then $$B_{k_i^-, k_i^+} \subseteq  B_{k_i^-+ k_i^+}$$ consists of those self-diffeomorphisms that take the set of the first $k_i^-$ punctures to itself.

A braid $b\in B_{k_i^-, k_i^+}$ gives a cobordism inside $D^2 \times [0,1]$. Taking the product of this cobordism with $S^1$, and using the identification between a neighborhood of $K_i \subset S^3$ and $S^1 \times D^2$, we get a cobordism $$\Sigma_b \subset S^1 \times D^2 \times [0,1] \subset S^3 \times [0,1]$$ from the cable $K(k^-,k^+)$ to itself. The associated cobordism map, $\KhR_N(\Sigma_b)$, gives an automorphism of $\KhR_N(K(k^-,k^+))$. In fact, the assignment $\beta_i: b\to \KhR_N(\Sigma_b)$ gives a group action on Khovanov-Rozansky homology
$$ \beta_i : B_{k_i^-, k_i^+}\to \Aut\bigl( \KhRN(K(k^-,k^+)) \bigr).$$

Let $e_i\in\Z^n$ denote the $i^{th}$ basis vector.  
Observe that two strands parallel to $K_i$, if they have opposite orientations, co-bound a ribbon band $R_i$ in $S^3$. By pushing $R_i$ into $S^3 \times [0,1]$ so that it is properly embedded there, removing a disk from $R_i$, and taking the disjoint union with the identity cobordisms on the other strands, we obtain an oriented cobordism $Z_i$ from $K(k^-,k^+) \sqcup U$ to $K(k^-+e_i, k^++e_i)$. Here, $U$ is the unknot and $\sqcup$ denotes split disjoint union. Observe that $\chi(Z_i)=-1$, and the framing on $K_i$ induces a framing on $Z_i$. By the discussion in Section~\ref{sec:framings}, there is a well-defined cobordism map
$$ \KhRN(Z_i) : \KhRN(K(k^-,k^+) \sqcup U) \to \KhRN(K(k^-+e_i, k^++e_i))$$
which changes the bi-grading by $(0,N-1)$. Note that
$$\KhRN(K(k^-,k^+) \sqcup U) \cong \KhRN(K(k^-,k^+)) \otimes \KhRN(U).$$
Recall from \eqref{eq:A} that $\A=\KhRN(U) \cong (\Z[X]/\langle X^N \rangle)\{1-N\}$.

Thus, the information in the map $\KhRN(Z_i)$ is encoded in $N$ maps
$$ \psi^{[m]}_i : \KhRN(K(k^-,k^+)) \to \KhRN(K(k^-+e_i, k^++e_i)), \ m=0,\dots, N-1$$
given by
\begin{equation}
\label{eq:phipsi}
 \psi^{[m]}_i(v) = \KhRN(Z_i)(v \otimes X^m).
 \end{equation}
Note that $\psi^{[m]}_i$ changes the bi-grading by $(0, 2m)$.

\begin{remark}
Let $\hZ_i$ be the cobordism from $K(k^-,k^+)$ to $K(k^-+e_i, k^++e_i)$ obtained from $Z_i$ by reintroducing the disk that was removed from the ribbon $R_i$. Since the map associated to a disk in Khovanov-Rozansky homology takes $1 \mapsto 1$, we see that $\psi^{[0]}_i = \KhRN(\hZ_i)$. More generally, $\psi^{[m]}_i$ is the cobordism map associated to $\hZ_i$ decorated with $m$ dots, in the sense of \cite[Example 2.1]{MWW}.
\end{remark}

\begin{remark}
For conciseness, we did not include the link $K$ and the values of $k^-$ and $k^+$ in the notation $\beta_i$, $R_i$, $Z_i$, $\psi^{[m]}_i$.
\end{remark}

Let $W$ be the $2$-handlebody obtained from $B^4$ by attaching handles along the link $K$. The  homology $H_2(W;\Z)$ is freely generated by the cores of the handles, capped with Seifert surfaces for each $K_i$. We identify $H_2(W;\Z) \cong \Z^n$ by letting the capped core of the $i$th handle correspond to $e_i$.

For $\alpha=(\alpha_1,\dots,\alpha_n)\in H_2(W;\Z) \cong\Z^n$, let $\alpha^+$ denote its positive part and $\alpha^-$ its negative part; i.e., $\alpha^+_i=\operatorname{max}(\alpha_i,0)$ and $\alpha^-_i=\operatorname{min}(\alpha_i,0)$. We also let $|\alpha| = \sum_i |\alpha_i|$. 

\begin{definition}
\label{def:cabled}
The cabled Khovanov-Rozansky homology of $K$ at level $\alpha$ is 
\[
\cKhRNa(K) =  \Bigl( \bigoplus\limits_{r\in\N^n} \KhRN(K(r-\alpha^-,r+\alpha^+))\{(1-N)(2r+|\alpha|)\} \Bigr)/ \sim
\]
where the equivalence $\sim$ is the transitive and linear closure of the relations
\begin{equation}
\label{eq:sim}
\beta_i(b)v \sim v, \ \ \psi^{[m]}_i(v) \sim 0 \text{ for } m < N-1,\ \ \psi^{[N-1]}_i(v) \sim v
\end{equation}
for all $i=1, \dots, n$; $b \in B_{r_i-\alpha^-, r_i+\alpha^+}$, and $v \in  \KhRN(K(r-\alpha^-,r+\alpha^+)).$
\end{definition}

Observe that the equivalence relation preserves the bi-grading, and hence there is an induced bi-grading on $\cKhRNa(K)$.

\begin{remark}\label{remark:ambiguity}
In principle, there are several different maps of the form $\psi_i^{[m]}$ (with the same domain and target), corresponding to different choices of the pair of oppositely oriented strands that co-bound $R_i$. However, these maps differ by post-composition with some $\beta_i(b)$. Therefore, when we divide by the relations \eqref{eq:sim}, which already include $\beta_i(b)v \sim v$, using one choice of $\psi_i^{[m]}$ is the same as using any other choice.
\end{remark}

\subsection{The cabled Khovanov homology}

When $N=2$, which is the case corresponding to Khovanov homology, Grigsby, Licata, and Wehrli \cite{GLW} showed that the braid group action on the cable of a knot factors through the symmetric group. In order to explain this fact, we first need to discuss orientations. Recall that Khovanov homology only depends on the orientation of a link through a shift in the quantum degree. This shift depends on the number of positive and negative crossings in the chosen orientation, which, for a cable of a knot, only depends on the number of strands oriented each way. Symmetries of $K$ induce maps on Khovanov homology, even if they are not symmetries of $K$ as an \emph{oriented} link.

Now, for a framed knot $K$, let $K(k^-,k^+)$ be an oriented cable with at least two components (for fewer components, the discussion of the braid group action is trivial). Number the strands so that the first $k^-$ are negatively oriented and the last $k^+$ are positively oriented. For $1\leq i< k^-+k^+$, choose a new orientation, if necessary, on $K(k^-,k^+)$ so that the $i^{\text{th}}$ and $(i+1)^{\text{st}}$ strands are oppositely oriented. Call this new oriented cable $K(\ell^-,\ell^+)$ , so that $\ell^-+\ell^+=k^-+k^+$ and $\ell^-,\ell^+\geq 1$. We have an isomorphism $\Kh(K(k^-,k^+))\cong\Kh(K(\ell^-,\ell^+))\{w\}$ for some shift $\{w\}$. Let $$\hZ(i):K(\ell^--1,\ell^+-1)\to K(\ell^-,\ell^+)$$ be the previously defined cobordism where we specify that $\hZ$ introduces the $i^{\text{th}}$ and $(i+1)^{\text{st}}$ components. Let $\hZ^r(i)$ be the reverse of $\hZ(i)$. Then, $\Kh(\hZ(i))\circ\Kh(\hZ^r(i))$ is an endomorphism of $\Kh(K(\ell^-,\ell^+))$. After canceling the shifts, it can also be viewed as an endomorphism of $\Kh(K(k^-,k^+))$. Let $\sigma_i$ be a generator of the braid group $B_{k^-+k^+}$ that interchanges the $i^{\text{th}}$ and $(i+1)^{\text{st}}$ strands. Then $\beta(\sigma_i)$ is an automorphism of $\Kh(K(k^-,k^+))$. With this understood, we have:
\begin{proposition}[Proposition 9 in \cite{GLW}]\label{prop:glw}
$$\beta(\sigma_i)=\operatorname{id}+\Kh(\hZ(i))\circ\Kh(\hZ^r(i))=\beta(\sigma_i^{-1})$$
as endomorphisms of $\Kh(K(k^-,k^+))$. In particular, the equation $\beta(\sigma_i)=\beta(\sigma_i^{-1})$ implies that $\beta$ factors through the symmetric group $S_{k^-+k^+}$.
\end{proposition}

\begin{remark}
Proposition~\ref{prop:glw} is stated in \cite{GLW} for the action of the braid group on sutured annular Khovanov homology, but the same proof works for Khovanov homology.
\end{remark}

Thus, when $K$ is a framed knot, we can replace the group $B_{r-\alpha^-, r+\alpha^+}$ in the definition of the cabled Khovanov homology by the product of symmetric groups, $S_{r-\alpha^-}\times S_{r+\alpha^+}$.

If we work over a ring $\k$ where $2$ is invertible, then we can use Proposition \ref{prop:glw} to simplify the cabled Khovanov-Rozansky homology even further. Let $\Kh(K;\k)$ denote the Khovanov homology of $K$ with coefficients in $\k$. We can define $\cKhR_{2,\alpha}(K;\k)$, the cabled Khovanov homology with coefficients in $\k$, by replacing the Khovanov homologies of the cables in Definition \ref{def:cabled} with their versions with coefficients in $\k$. With this notation, we have:

\begin{proposition}
Let $\k$ be a commutative ring where $2$ is invertible. Then the cabled Khovanov homology over $\k$ of a framed knot $K$ at level $\alpha$ is

\[
\cKhR_{2,\alpha}(K;\k) =  \Bigl( \bigoplus\limits_{r\in\N} \Kh(K(r-\alpha^-,r+\alpha^+);\k)\{-2r-|\alpha|\} \Bigr)/ \sim
\]
where the equivalence $\sim$ is the transitive and linear closure of the relations
$$ \beta(b)v \sim v, \ \psi^{[1]}(v)\sim v $$
for all $b \in S_{2r+|\alpha|}$, and for all $v \in  \Kh(K(r-\alpha^-,r+\alpha^+);\k).$
\end{proposition}
\begin{proof}
The equivalence relation defining $\cKhR_{2,\alpha}(K;\k)$ consists of the relations
\begin{equation}
\beta(b)v \sim v, \ \ \psi^{[0]}(v) \sim 0,\ \ \psi^{[1]}(v) \sim v
\end{equation}
for all $b \in B_{r-\alpha^-, r+\alpha^+}$, and $v \in  \Kh(K(r-\alpha^-,r+\alpha^+);\k)$. The group action $\beta$ factors through $S_{r-\alpha^-}\times S_{r+\alpha^+}$. Let $\sigma_i\in S_{2r+|\alpha|}$ denote the transposition that interchanges the $i^{\text{th}}$ and $(i+1)^{\text{st}}$ elements. If we number the components of $K(r-\alpha^-,r+\alpha^+)$ so that the first $r-\alpha^-$ are negatively oriented and the rest positively oriented, then the $\sigma_i$ for $i=1,\dots, r-\alpha^--1,r-\alpha^-+1,\dots,2r+|\alpha|-1$ generate the desired subgroup $S_{r-\alpha^-}\times S_{r+\alpha^+}$. Thus, it suffices to take the relations $\beta(\sigma_i)v \sim v$ for all such $i$.

The map $\psi^{[0]}$ is associated to a cobordism $\hZ$ that introduces two oppositely oriented components. As explained in Remark \ref{remark:ambiguity}, we can pick any two oppositely oriented components. So, we choose the components numbered $r-\alpha^-$ and $r-\alpha^-+1$. Then, by Proposition \ref{prop:glw}, we have
\begin{equation}\label{eq:glw}
\beta(\sigma_{r-\alpha^-})=\operatorname{id}+\psi^{[0]}\circ \Kh(\hZ^r).
\end{equation}
Observe that $\hZ^r\circ \hZ$ is the union of a torus and some cylinders. By \cite{GujralLevine}, we have $$\Kh(\hZ^r)\circ \Kh(\hZ)=2\operatorname{id}.$$ Since we are working over a ring where $2$ is a unit, this map is an isomorphism. In particular, $\Kh(\hZ^r)$ is surjective. Thus, by Equation \eqref{eq:glw}, the relations $\psi^{[0]}(w) \sim 0$ for all $w$ are equivalent to the relations $\beta(\sigma_{r-\alpha^-})v\sim v$ for all $v$. But since the relations defining $\cKhR_{2,\alpha}(K;\k)$ already include $\beta(\sigma_i)v\sim v$ for $i\neq r-\alpha^-$, adding the relation $\beta(\sigma_{r-\alpha^-})v\sim v$ generates the full symmetric group $S_{2r+|\alpha|}$. 
\end{proof}

\section{2-Handlebodies}
\label{sec:2h}
In this section we prove our main result, about the skein lasagna modules of 2-handlebodies. 

\begin{proof}[Proof of Theorem~\ref{2handles}]
We first define a map
\begin{align*}
\widetilde{\Phi} : \bigoplus\limits_{r\in\N^n} \KhRN(K(r-\alpha^-,r+\alpha^+))\{(1-N)(2r+|\alpha|) \}  \to \Sz(W;\emptyset,\alpha), \ \ \ \widetilde{\Phi}(v) = [F_v]
\end{align*}
as follows. Let $B$ be $4$-dimensional ball slightly smaller than the $0$-handle, contained in the interior of that handle. Given an element $v\in \KhRN(K(r-\alpha^-,r+\alpha^+))$, we define $\widetilde{\Phi}(v)$ to be the class of the lasagna filling $F_v$ with $B$ as the only input ball, with $B$ decorated with the framed, oriented link $K(r-\alpha^-,r+\alpha^+)$ and labeled by the element $v$, and with the surface given by the disjoint union of $r_i-\alpha_i^-$ negatively oriented discs parallel to the core of $i^{th}$ 2-handle and $r_i+\alpha_i^+$ positively oriented such discs (union over all $i$). We will denote these disks by $C_j^{i,\pm}$, where $1\leq j\leq r_i\pm\alpha_i^{\pm}$. Since the disks are contractible, they have unique framings. The homology class of this surface in $H_2(W,B;\Z) \cong H_2(W;\Z)$ is clearly $\alpha$ by construction. See Figure~\ref{fig:generator}.

\begin {figure}
\begin {center}
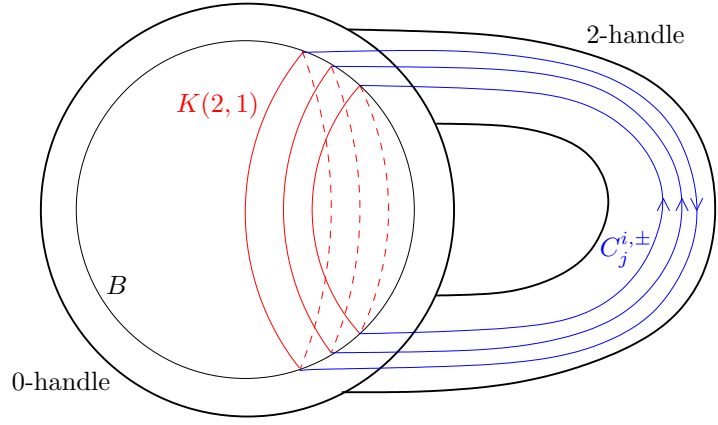
\caption {A generator for $\Sz(W; \emptyset, 0)$.}
\label{fig:generator}
\end {center}
\end {figure}

We claim that, under the equivalence relation from Definition~\ref{def:cabled}, $\widetilde{\Phi}$ maps equivalent elements to the same class in $\Sz(W;\emptyset,\alpha)$. There are three types of relations to be checked.

First, consider the braid group action. Intuitively, this permutes the disks $C_j^{i,\pm}$ in the $2$-handle. More precisely, let $\Sigma_b$ be the cobordism from $K(r-\alpha^-,r+\alpha^+)$ to itself associated to a braid $b \in B_{k_i^-, k_i^+}$. We can then view the filling $F_v$ of $W$ as obtained from $F_{\beta_i(b)v}$ by inserting into $B$ a filling made of a smaller ball $B'$ and the surface $\Sigma_b$, with the input labeled by $v$. Therefore, $F_v$ and $F_{\beta_i(b)v}$ represent the same class in $\Sz(W;\emptyset,\alpha)$. 

Second, consider a slightly smaller ball $B'$ contained in $B$, so that the region between $B'$ and $B$ is a copy of $S^3 \times [0,1]$. We put in that region the cobordism $D \subset S^3 \times [0,1]$, from $K(r-\alpha^-,r+\alpha^+) \sqcup U$ to $K(r-\alpha^-, r+\alpha^+)$, which is simply the split disjoint union of the identity on $K(r-\alpha^-,r+\alpha^+)$ with a disk capping the unknot $U$. The cobordism map associated to the cap is the counit \eqref{eq:counit}. Therefore,
\begin{equation}
\label{eq:KhD}
 \KhRN(D)(v \otimes X^m)=0 \text{ for } m < N-1; \ \ \ \  \KhRN(D)(v \otimes X^{N-1}) = v.
 \end{equation}
For every $w \in \KhRN(K(r-\alpha^-,r+\alpha^+) \sqcup U)$, we construct a lasagna filling $E_w$ with input ball $B'$, surface 
$$D \cup \bigcup_{i, j} C^{i, -}_j \cup \bigcup_{i, j} C^{i, +}_j $$
and label $w$. Thus, $E_w$ is obtained from the filling $F_{\KhRN(D)(w)}$ by adjoining $D$. It follows that the fillings $E_w$ and $F_{\KhRN(D)(w)}$ are equivalent.

On the other hand, as in Section~\ref{sec:cabled}, we also have a cobordism $Z_i \subset S^3 \times [0,1]$ from $K(r-\alpha^-,r+\alpha^+) \sqcup U$ to $K(r-\alpha^-+e_i, r+\alpha^++e_i)$. For every $w \in \KhRN(K(r-\alpha^-,r+\alpha^+) \sqcup U)$, we construct a new lasagna filling $E'_w$ from $F_{\KhRN(Z_i)(w)}$ by adjoining $Z_i$ in the region between $B'$ and $B$. Then, the fillings $E'_w$ and $F_{\KhRN(Z_i)(w)}$ are equivalent.

\begin {figure}
\begin {center}
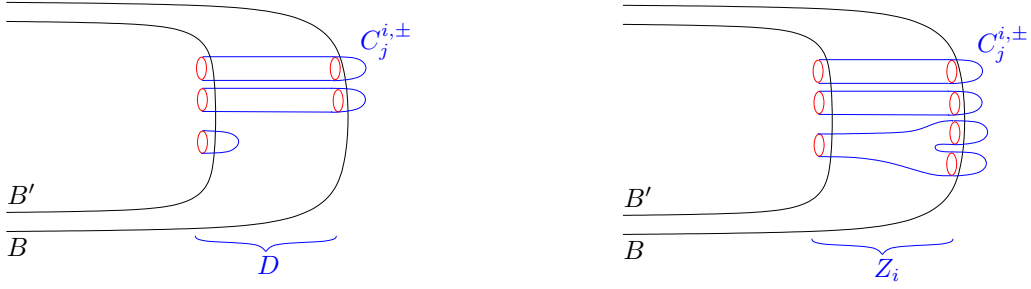
\caption {A schematic picture of the fillings $E_w$ (left) and $E'_w$ (right). For simplicity, we drew the components of the cable side by side, rather than nested.}
\label{fig:Ew}
\end {center}
\end {figure}

Note that $E'_w$ has the same input data as $E_w$ (namely, the ball $B'$ and the label $w$); see Figure~\ref{fig:Ew}. Moreover, the surface of $E'_w$ is obtained from that of $E_w$ by taking connected sum with the  closed surface
\begin{equation}
\label{eq:sphere}
C_{r_i-\alpha_i^-+e_i}^{i,-}\cup R_i \cup C_{r_i+\alpha_i^++e_i}^{i,+} \; \cong \; S^2,
\end{equation}
where $R_i$ is the ribbon between two oppositely oriented copies of $K_i$, as in Section~\ref{sec:cabled}. The copy of $S^2$ from \eqref{eq:sphere} can be isotoped to lie entirely in the $2$-handle. Since this sphere bounds a 3-ball, taking a connected sum with it can be viewed simply as a surface isotopy. Therefore, $E_w$ and $E'_w$ are equivalent lasagna fillings. We conclude that
$$ \widetilde{\Phi}(\KhRN(D)(w)) = [F_{\KhRN(D)(w)}] = [E_w] = [E'_w] = [F_{\KhRN(Z_i)(w)}] = \widetilde{\Phi}(\KhRN(Z_i)(w)),$$
for every $w \in \KhRN(K(r-\alpha^-,r+\alpha^+) \sqcup U)\cong \KhRN(K(r-\alpha^-,r+\alpha^+)) \otimes \A.$ 

Let $v \in \KhRN(K(r-\alpha^-,r+\alpha^+))$. If we take $w=v \otimes X^m$, in view of Equations~\eqref{eq:phipsi} and \eqref{eq:KhD}, we have
$$0=\widetilde{\Phi}(\psi^{[m]}_i(v)) \text{ for } m <  N-1; \ \ \ \ \widetilde{\Phi}(v) = \widetilde{\Phi}(\psi^{[N-1]}_i(v)).$$

We have now verified our claim that $\widetilde{\Phi}$ takes equivalent elements to the same equivalence class of lasagna fillings. This shows that $\widetilde{\Phi}$ descends to a map
$$\Phi: \cKhRNa(K) \to \Sz(W;\emptyset,\alpha).$$

Next, we define an inverse $\Phi^{-1}$ to $\Phi$. Let $F$ be a lasagna filling with surface $\Sigma$. By an isotopy, we can push the input balls of $F$ into the interior of the $0$-handle and we can arrange that $\Sigma$ intersects the cocores $G_i$ of the $2$-handles transversely in a number of points. Since we assume that the relative homology class of $\Sigma$ is $\alpha$, the signed intersection of $\Sigma$ with $G_i$ must be $\alpha_i$. Thus, $\Sigma$ intersects $G_i$ in $|\alpha_i|+2r_i$ points. After another isotopy, we can assume that $\Sigma$ intersects the $2$-handles only in core-parallel disks, one for each intersection with the cocores. That is, all interesting topology of $\Sigma$ is pushed into the $0$-handle. Next, choose a ball $B'$ slightly smaller than the $0$-handle, so that it contains all input balls and all of the interesting topology of $\Sigma$. Finally, evaluate the part of $F$ inside of $B'$. These modifications show that $F$ is equivalent to a filling of the form $\widetilde{\Phi}(v)$, and we let $$\Phi^{-1}([F]) := [v]$$
for $v\in \KhRN(K(r-\alpha^-,r+\alpha^+))$. Let us check that $\Phi^{-1}$ is well-defined. 

Firstly, if we change a lasagna filling $F$ by filling in an input ball in $B^4$, this does not change its its image under $\Phi^{-1}$, because the intersection with the cocores of the $2$-handles  is unchanged. 

Secondly, in the definition of $\Phi^{-1}$ we chose an isotopy of $\Sigma$ that makes it transverse to the cocores $G_i$ of the $2$-handles. If we made a different choice, the isotopy relating the two choices is a $1$-parameter family of surfaces $\Sigma_t, t \in [0,1]$. Let
$$ Y = \bigcup_{t \in [0,1]} \bigl( \{t\} \times \Sigma_t \bigr) \subset [0,1] \times W. $$
We can assume that $Y$ is a smooth $3$-dimensional submanifold with boundary, and that $Y$ intersects each $[0,1] \times G_i$ transversely in a $1$-manifold $P_i$. Let
$$ \pi_i: P_i \to [0,1]$$
be the composition of the inclusion into $[0,1]\times W$ with projection to the $[0,1]$ factor. After an isotopy of $Y$ rel boundary, we arrange so that $\pi_i$ are local diffeomorphisms away from finitely many critical points (caps and cups). The critical values $t_1 < \dots < t_n$ of $\pi_i$, together with the endpoints $t_0=0$ and $t_{n+1} =1$, split $[0,1]$ into finitely many intervals of the form $[t_k, t_{k+1}]$. After another isotopy of $Y$, we can assume that for some small $\epsilon > 0$, the collections of points $\pi_i^{-1}(t_k + \epsilon)$ and $\pi_i^{-1}(t_{k+1} - \epsilon)$ coincide, as oriented submanifolds of the cocore $G_i$. Thus, on the interval $[t_k + \epsilon, t_{k+1} -\epsilon]$, the surfaces $\Sigma_t$ stay transverse to the cocores, and the only effect of varying the lasagna filling is to replace the element $v$ with $\beta_i(b)(v)$ for some braid $b \in B_{r_i-\alpha_i^-, r_i+\alpha_i^+}.$ However, we have $\beta_i(b)(v) \sim v$, so the value of $\Phi^{-1}([F])$ is unchanged. 

At the critical values of $\pi_i$, the surfaces $\Sigma_t$ are no longer transverse to the cocores of the $2$-handles. Rather, what happens is that we introduce or remove two intersections of opposite signs. Introducing such points corresponds to ``pushing a disk'' from one lasagna filling into the cocore; that is, replacing a lasagna filling of the form $E_w$ with one of the form $E'_w$, for some $w \in \KhRN(K(r-\alpha^-,r+\alpha^+) \sqcup U)$. The corresponding values of $\Phi^{-1}$ for these fillings are $F_{\KhRN(D)(w)}$ and $F_{\KhRN(Z_i)(w)}$. When $w = v \otimes X^m$ for $m < N-1$, these give $0$ and $\psi^{[m]}_i(v)$, and when $w=v \otimes X^{N-1}$, they give $v$ and $\psi^{[N-1]}_i(v)$. Since the equivalence that gives the cabled Khovanov-Rozansky homology includes the relations 
$$\psi^{[m]}_i(v) \sim 0 \text{ for } m < N-1,\ \ \psi^{[N-1]}_i(v) \sim v,$$
we see that the value of $\Phi^{-1}([F])$ is the same for isotopic lasagna fillings. 

This completes the proof that $\Phi^{-1}$ is well-defined. It is straightforward to check that $\Phi$ and $\Phi^{-1}$, as defined above, are inverse to each other.
\end{proof}

\begin{remark}
Theorem~\ref{2handles} implies that $\cKhRNa(K)$ is invariant under handleslides among the components of the link $K$. One can also give a more direct proof of this fact, using just the functoriality of Khovanov-Rozansky homology, and without reference to the skein algebra. We leave this as an exercise for the reader.
\end{remark}

\section{The $0$-framed unknot}
\label{sec:D0}

The ring structure on $\Sz(S^2 \times D^2; \emptyset)$ is given by taking the union of lasagna fillings. More precisely, let $I\subset \partial D^2$ be an interval in the boundary of the disc. Then we have a gluing $$(S^2\times D^2) \cup_{S^2\times I}(S^2\times D^2)\cong S^2\times D^2.$$
This decomposition allows us to define a map $$m: \Sz(S^2 \times D^2; \emptyset)\otimes_{\Z} \Sz(S^2 \times D^2; \emptyset)\to \Sz(S^2 \times D^2; \emptyset)$$ by the formula $m(F_1\otimes F_2)=F_1\cup F_2$. The fact that isotopic lasagna fillings are equivalent shows that this map endows $\Sz(S^2 \times D^2; \emptyset)$ with the structure of an associative, commutative algebra. It has unit given by the empty filling. Before proving Theorem~\ref{thm:S2D2}, we identify the braid group action on the Khovanov-Rozansky homology of the unlink.
\begin{lemma}\label{braidgroupunlink}
The braid group action on the Khovanov-Rozansky homology of the unlink factors through the symmetric group.
\end{lemma}
\begin{proof}
Let $\sigma_i\in B_n$ be a generator of the braid group $B_n$ and $\Sigma_{\sigma_i}=\sigma_i\times S^1$ the associated cobordism from the $n$-component unlink $U(n)$ to itself. By definition, the action of $\sigma_i$ on $\KhR_N(U(n))$ is given by $$\beta(\sigma_i)=\KhR_N(\Sigma_{\sigma_i}).$$ To show that $\beta$ factors through the symmetric group, it suffices to check that $\beta(\sigma_i^2)$ is the identity. Note that there is an action of $\A^{\otimes n}$ on $\KhR_N(U(n))$ induced by the identity cobordism on $U(n)$ decorated with dots on each component, as in \cite[Example 2.1]{MWW}. The cobordism map $\beta(\sigma_i^2)$ is an isomorphism of $\A^{\otimes n}$-modules and $\KhR_N(U(n))$ is a rank one $\A^{\otimes n}$-module, so to compute $\beta(\sigma_i^2)$ it suffices to determine the image of $1\in\KhR_N(U(n))$. But $1$ must be sent to either $\pm 1$, so $\beta(\sigma_i^2)=\pm \operatorname{id}$. To determine the sign, since $\KhR_N(U(n))$ has no $2$-torsion, we just need to evaluate $\beta(\sigma_i^2)$ on any non-zero element.

Let $c_i$ be the $(n-2,n)$-tangle with a cup between the $i^{\text{th}}$ and $(i+1)^{\text{st}}$ endpoints. Then, $\sigma_i^2\circ c_i$ is isotopic to $c_i$. Let $C_i=c_i\times S^1$ be the associated cobordism from $U(n-2)$ to $U(n)$. Then $\KhR_N(C_i)(1)=1\otimes X_i + X_{i+1}\otimes 1$, where $X_i$ is the generator of the $i^{\text{th}}$ tensor factor of $\A^{\otimes n}$. In particular, $\KhR_N(C_i)$ is not the zero map.  Since $\sigma_i^2\circ c_i \sim c_i$, we have $\Sigma_{\sigma_i^2}\circ C_i\sim C_i$, so $\beta(\sigma_i^2)$ is the identity on the image of $\KhR_N(C_i)$. By the previous discussion, it must be the identity map.
\end{proof}

\begin{proof}[Proof of Theorem~\ref{thm:S2D2}]
Since $S^2 \times D^2$ is the result of attaching a $2$-handle along a $0$-framed unknot, we can apply Theorem~\ref{2handles} to obtain a group isomorphism $$\Sz(S^2 \times D^2; \emptyset)\cong \bigoplus\limits_{\alpha\in Z} \cKhRNa(U, 0), $$
and from Definition~\ref{def:cabled} we have
\[
\cKhRNa(U, 0) =  \Bigl( \bigoplus\limits_{r\in\N} \KhRN(U(r-\alpha^-,r+\alpha^+))\{(1-N)(2r+|\alpha|) \} \Bigr)/ \sim
\]
where $U(r-\alpha^-,r+\alpha^+)$ is the $(2r+|\alpha|)$-component unlink. The Khovanov homology of the $n$-component unlink is $\A^{\otimes n}$, where $\A=(\Z[X]/\langle X^N \rangle)\{1-N\}$. Letting $\mathcal{V}=\A\{1-N\}$ in order to absorb the shifts in the definition of $\cKhRNa$, we have
\[
\cKhR_N(U, 0) =  \bigoplus\limits_{\alpha\in\Z} \Bigl( \bigoplus\limits_{r\in\N} \mathcal{V}^{\otimes (r-\alpha^-)}\otimes \mathcal{V}^{\otimes (r+\alpha^+)} \Bigr)/ \sim
\]
The proof of the isomorphism in Theorem ~\ref{2handles} shows that the disjoint union of lasagna fillings  corresponds to the tensor product of elements in $\cKhR_N(U, 0)$. That is, the algebra structure on $\cKhR_N(U, 0)$ is defined as follows: for $v_i\in \mathcal{V}^{\otimes (r_i-\alpha_i^-)}$ and $w_i\in \mathcal{V}^{\otimes (r_i+\alpha_i^+)}$, $i=1,2$, we have $$(v_1\otimes w_1)\cdot (v_2\otimes w_2)=(v_1\otimes v_2)\otimes (w_1\otimes w_2)\in \mathcal{V}^{\otimes (r_1+r_2-\alpha_1^--\alpha_2^-)}\otimes \mathcal{V}^{\otimes (r_1+r_2+\alpha_1^++\alpha_2^+)}.$$

The group $\cKhRNa(U, 0)$ is supported in homological degree $0$ because $\mathcal{V}$ is in homological degree $0$. By Lemma~\ref{braidgroupunlink}, the braid group action factors through the symmetric group. This action by $S_{r-\alpha^-}\times S_{r+\alpha^+}$ simply permutes the tensor factors of $\mathcal{V}^{\otimes (r-\alpha^-)}\otimes \mathcal{V}^{\otimes (r+\alpha^+)}$. Reducing modulo this action, we are left with 
\[
\cKhRNa(U, 0) =  \Bigl( \bigoplus\limits_{r\in\N} \operatorname{Sym}^{r-\alpha^-}(\mathcal{V})\otimes \operatorname{Sym}^{r+\alpha^+}(\mathcal{V}) \Bigr)/ \sim,
\]
where $\operatorname{Sym}^r(\mathcal{V})$ is the $r^{th}$ symmetric power of $\mathcal{V}$. We write basis elements of $\mathcal{V}$ with indices lowered to avoid confusion between the algebra structure on $\operatorname{Sym}(\mathcal{V})$ and the algebra structure on $\A$ (the latter plays no role in this discussion). For convenience, we also re-index the basis and write $x_k\in \mathcal{V}$ for $X^{N-1-k}\in\A$. Note that $x_k$ has quantum degree $-2k$.

We now consider the maps $\psi^{[m]}$ in the equivalence relation $\sim$. For the 0-framed unknot, the maps $\psi^{[m]}$ are given by multiplication by $\Delta(X^m)$, which is defined in Equation~\eqref{eq:comult}. In our preferred basis for $\mathcal{V}$, we have $$\Delta(x_m)=\sum\limits_{k=0}^m x_k\otimes x_{m-k}.$$
We can use the multiplicative structure on $\cKhR_N(U,0)$ to write $\psi^{[N-1-m]}(v)=\Delta(x_m)\cdot v$. So, if we define $\mathcal{I}$ to be the ideal of $\cKhR_N(U, 0)$ generated by $\Delta(x_0)-1$ and the $\Delta(x_m)$ for $0< m\leq N-1$, then we have a ring isomorphism 
$$\cKhR_N(U, 0)\cong \bigoplus\limits_{\alpha\in\Z} \Bigl( \bigoplus\limits_{r\in\N} \operatorname{Sym}^{r-\alpha^-}(\mathcal{V})\otimes \operatorname{Sym}^{r+\alpha^+}(\mathcal{V}) \Bigr)/ \mathcal{I}\  \cong \bigl( \Sym^*(\mathcal{V}) \otimes  \Sym^*(\mathcal{V}) \bigr) / \mathcal{I}.$$
To conclude, we must show that the above ring is isomorphic to $\Z[A_1,\dots, A_{N-1},A_0,A_0^{-1}]$. Observe that $\Sym^*(\mathcal{V}) \otimes  \Sym^*(\mathcal{V})$ is freely generated as a commutative algebra by the elements
\begin{align*}
& A_i=1\otimes x_i, \text{ and }B_i= x_i\otimes 1 \text{ for } 0\leq i\leq N-1.
\end{align*}
In terms of these ring generators, the generators of the ideal $\mathcal{I}$ are
\begin{align*}
&\Delta(x_0)-1 = A_0B_0-1 \text{ and }\\
&\Delta(x_m)=\sum\limits_{k=0}^m A_kB_{m-k} \text{ for } 0< m\leq N-1.
\end{align*}
Therefore, $$\cKhR_N(U,0)\cong \Z[A_0,\dots, A_{N-1},B_0,\dots,B_{N-1}]/(A_0B_0-1, \sum\limits_{k=0}^m A_kB_{m-k}) \text{ for } 0< m\leq N-1.$$
The relation $A_0B_0-1$ lets us write $B_0=A_0^{-1}$. Then, since $A_0$ is invertible, we can replace the generator $\sum\limits_{k=0}^m A_kB_{m-k}$ by $$A_0^{-1}(\sum\limits_{k=0}^m A_kB_{m-k})=B_m+A_0^{-1}(\sum\limits_{k=1}^m A_kB_{m-k}).$$
This relation allows us to write the generator $B_m$ in terms of the $A_i$ and the $B_j$ with $0<j<m$. That is, the reducing modulo $\Delta(x_{N-1})$ is equivalent to getting rid of the generator $B_{N-1}$. Then, reducing modulo $\Delta(x_{N-2})$ is equivalent to removing the generator $B_{N-2}$, and so on until we have used all of the $\Delta(x_m)$ relations and removed all of the $B_m$ generators. This leaves us with
$$\cKhR_N(U,0)\cong\Z[A_1,\dots, A_{N-1},A_0,A_0^{-1}]$$
as advertised.

Since the relatively homology class of the lasagna filling associated to some $p\otimes q\in \Sym^{d_1}(\mathcal{V}) \otimes  \Sym^{d_2}(\mathcal{V})$ is given by $d_2-d_1$, we see that the $A_i$ have homology class $\alpha=1$. Also, the homology class is additive under multiplication in the ring. So, the elements with homology class $\alpha$ are precisely the homogeneous polynomials of degree $\alpha$.
\end{proof}

\section{The unknot with non-zero framing}
\label{sec:Dp}
The goal of this section is to prove Theorem~\ref{thm:Dp}.  This will be a direct consequence of Theorem~\ref{2handles} and the following:

\begin{proposition}\label{unknotp}
Let $(U,p)$ be the unknot with framing $p$. 

$(a)$ For $N=2$ and $p>0$, the cabled Khovanov-Rozansky homology of $(U,p)$ in homological degree $0$ and at level $0\in \Z$ is given by
$$\cKhR_{2,0}^{0, j}(U,p) = 0, \ \ \forall j\in \Z.$$

$(b)$ On the other hand, for $N=2$ and $p < 0$ we have
$$\cKhR_{2, 0}^{0, j}(U,p) = \begin{cases}
\Z &\text{if } j=0\\
0 & \text{otherwise.}
\end{cases}$$
\end{proposition}
 
 An outline of the proof of Proposition~\ref{unknotp} is as follows. The group $\cKhR_{2, 0}^{0, *}(U,p)$ is built from the Khovanov homologies of cables on the $p$-framed unknot. These cables are the $T_{2n,2np}$ torus links. The Khovanov homology of these links in homological degree $0$ has been computed and, when $p<0$, shown to be isomorphic to the center of Khovanov's arc algebra $H^n$ \cite{stosic}. For $p>0$, it is isomorphic to the dual of the center of the arc algebra. To explain these results, we introduce the arc algebra in Section \ref{sec:Hn} and describe explicit presentations of its center and the dual of its center in Sections \ref{sec:ZHn} and \ref{sec:dual}. In Section \ref{sec:hochschild}, we explain the connection between the center of the arc algebra and the Khovanov homology of torus links. In order to compute the cabled Khovanov homology, we also need to identify the cobordism maps and braid group actions appearing in Definition \ref{def:cabled}. In Section \ref{sec:cobordismmaps}, we compute these maps under the identification of the Khovanov homology of the torus links with the center of the arc algebra. We combine these ingredients in Section \ref{sec:unknotpproof} to prove Proposition~\ref{unknotp}.
 
 \subsection{The ring $H^n$ and tangle invariants}
 \label{sec:Hn}
In \cite{KhTangles}, Khovanov extended his theory $\Kh$ to tangles. We briefly sketch his construction here.

Throughout this section $N=2$, and the invariant of the unknot is
$$\A = (\Z[X]/\langle X^2 \rangle)\{1\} =\text{Span}\{1, X\}.$$
Let $\Cr_{n}$ be the set of crossingless matchings between $2n$ points on a line. For example, the following is an element of $\Cr_3$:
\[ \begin{picture}(0,0)%
\includegraphics{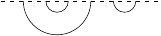}%
\end{picture}%
\setlength{\unitlength}{2368sp}%
\begingroup\makeatletter\ifx\SetFigFont\undefined%
\gdef\SetFigFont#1#2#3#4#5{%
  \reset@font\fontsize{#1}{#2pt}%
  \fontfamily{#3}\fontseries{#4}\fontshape{#5}%
  \selectfont}%
\fi\endgroup%
\begin{picture}(2124,471)(889,-1420)
\end{picture}%

\]
For $a \in \Cr_{n}$, we let $\bar a$ denote its reflection in the dashed line. Then, for every $a,b \in \Cr_n$, the composition $\bar b a$ is a collection of $k$ circles in the plane:
\[ 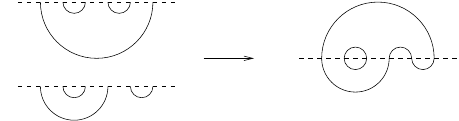
\]

To $\bar b a$ we associate the tensor product of copies of $\A$, one for each circle:
$$ \F(\bar b a) = \A^{\otimes k}.$$

Khovanov constructs a finite-dimensional graded ring
$$ H^n := \bigoplus_{a, b \in \Cr_n} \Hn{a}{b} $$
where
$$ \Hn{a}{b}= \F(\bar b a)  \{n\} .$$
The multiplication on $H^n$ is
$$ \Hn{a}{b} \otimes \Hn{d}{c} \to 0 \text{ if } b \neq d$$
and
$$\Hn{a}{b} \otimes \Hn{b}{c} \to \Hn{a}{c}$$
is given by a sequence of saddle maps taking $\bar b b$ into the identity tangle, using the multiplication and comultiplication maps on the Frobenius algebra $\A$:
\[ 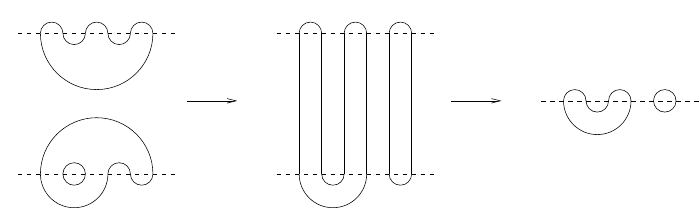
\]

Consider now a tangle $T$ represented by a diagram inside a rectangle, connecting $2n$ points at the bottom to $2m$ points at the top. This is called an $(m, n)$-tangle in the terminology of \cite{KhTangles}. To $T$ Khovanov associates a complex of $(H^n, H^m)$-bimodules 
\begin{equation}
\label{eq:FT}
 \F(T)  = \bigoplus_{a \in \Cr_n} \bigoplus_{b \in \Cr_m} {}_{a}(\F(T))_{b}
 \end{equation}
where ${}_{a}(\F(T))_{b}$ is the usual Khovanov complex associated to the link obtained from the composition $\bar b T a$. Different diagrams for the same tangle produce homotopy equivalent complexes.

Finally, let us discuss a module action on the rings $H^n$. This is the analogue of the action of 
$$R = \Z[X]/\langle X^2 \rangle = \A\{-1\}$$ on the Khovanov homology of a link $L$, which was constructed in  \cite{KhovanovPatterns}. The action of $r \in R$ is given by introducing a small unknot near a basepoint on the knot, marking it with $r$, and applying the multiplication map induced by the saddle cobordism from $L \sqcup U$ to $L$:
\[ 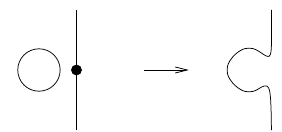
\]
More generally, if $L$ has $m$ components, there is an action of $R^{\otimes m}$ by using unknots near basepoints on each component; cf. \cite[Section 2]{HeddenNi} and \cite[Section 2.2]{BLS}.  

In our case, we view the ring $H^n$ as an algebra over 
\begin{equation}
\label{eq:R2n}
 R^{\otimes 2n} = \Z[X_1, \dots, X_{2n}]/\langle X_1^2, \dots, X_{2n}^2 \rangle
\end{equation}
by using unknots near each of the $2n$ points on the line. We set $X_i$ to be $(-1)^i$ times the $X$ action from the unknot near the $i$th point.

If $T$ is an $(m, n)$-tangle, then the complex of bimodules $\F(T)$ gets induced actions of $R^{\otimes 2n}$ (from the points at the bottom) and of $R^{\otimes 2m}$ (from the points at the top).

\begin{remark}
The analogues of the $H^n$ rings for the $\gl_2$ theory $\KhR_2$ were constructed by Ehrig, Stroppel and Tubbenhauer in \cite{EST1}, \cite{EST2}. It would be more natural to work with them, because the skein lasagna algebras are defined in terms of $\KhR_2$. However, we chose to use the original $H^n$ in order to be able to use various results from the literature that were proved in that context.
\end{remark}

\subsection{The center of $H^n$}\label{sec:ZHn}
We will be interested in the center of the ring $H^n$, which was computed in \cite{KhSpringer}.
\begin{theorem}[Khovanov \cite{KhSpringer}]
\label{thm:ZHn}
The center $Z(H^n)$ is isomorphic to the polynomial ring $\Z[X_1, \dots, X_{2n}]$ modulo the ideal generated by the elements
$$ X_i^2, \ \ i=1, \dots, 2n$$
and
$$  \sum_{|I| =k} X_I, \ \ k=1, \dots, 2n,$$
where $X_I = X_{i_1}\dots X_{i_k}$ for $I=\{i_1, \dots, i_k\}$, and the sum is over all the cardinality $k$ subsets of $\{1, \dots, 2n\}$.
\end{theorem}

We grade $Z(H^n)$ so that each $X_i$ is in degree $2$. This corresponds to the convention for $\KhR_2$, and is opposite the convention for $\Kh$; see Section~\ref{sec:framings}.

One can read off from \cite{KhSpringer} an explicit description of the elements $X_i$. Let $p_1, \dots, p_{2n}$ be the $2n$ points on the line (in this order) that we connect by crossingless matchings. Then, we have
$$X_i = \sum_{a \in \Cr_n} {}_a (X_i)_a, \ \ \ {}_a (X_i)_a\in \Hn{a}{a} \cong \A^{\otimes k}\{n\}$$
where ${}_a (X_i)_a$ is the tensor product of $1 \in \A$ for each circle not going through $p_i$, and of $(-1)^i X \in \A$ for the circle going through $p_i$. In other words, $X_i$ exactly correspond to the variables in the ring $R^{\otimes 2n}$ from \eqref{eq:R2n}, applied to the identity element $1 \in Z(H^n) \subset H^n$. Thus, we can improve Theorem~\ref{thm:ZHn} to a statement about $R^{\otimes 2n}$-algebras:
\begin{proposition}
\label{prop:ZHn}
As a $R^{\otimes 2n}$-algebra, the center $Z(H^n)$ is isomorphic to
$$ R^{\otimes 2n}/ \bigl \langle \sum_{|I| =k} X_I, \ \ k=1, \dots, 2n\bigr \rangle$$
where the sum is over all the cardinality $k$ subsets of $\{1, \dots, 2n\}$.
\end{proposition}

\begin{lemma}
The degree $2k$ part of $Z(H^n)$, denoted $Z(H^n)_{2k}$, is generated (as an Abelian group) by the $X_I$ over all cardinality $k$ subsets $I\subseteq \{1, \dots, 2n\}$, subject to the linear relations:
\begin{equation}
\label{eq:ZHnj}
 \sum_{|I | = k, \ I \supseteq J} X_I = 0,
\end{equation}
for every subset $J \subseteq \{1, \dots, 2n\}$ of cardinality $|J| <k$.
\end{lemma}

\begin{proof}
Starting from the description in Theorem~\ref{thm:ZHn}, we see that $Z(H^n)_{2k}$ is generated by all monomials of degree $k$ modulo the homogeneous relations of polynomial degree $k$. Such relations are of the form
\begin{equation}\label{eq:rel1}
X_J\sum_{|I| =k-\ell} X_I
\end{equation}
for all $0\leq \ell < k$ and all $J$ with $|J|=\ell$. Using that $X_i^2=0$, we can rewrite \eqref{eq:rel1} as
\begin{equation}\label{eq:rel2}
X_J\sum_{\substack{|I| =k-\ell\\ I\cap J=\emptyset}} X_I
\end{equation}
Replacing $I$ by $I\cup J$ in \eqref{eq:rel2}, we get the desired conclusion. \end{proof}

Sto\v{s}i\'{c}  \cite[Proposition 1]{stosic} showed that
\begin{equation}
\label{eq:choose}
Z(H^n)_{j} \cong \begin{cases} \Z^{{2n \choose k} - {2n \choose k-1}}& \text{if } j=2k, \ k=0, \dots, n,\\
0 & \text{otherwise.}
\end{cases}
\end{equation}
The total rank of $Z(H^n)$ is ${2n \choose n}$. We can give a concrete basis for $Z(H^n)$ as follows.

\begin{definition}
A subset $I \subseteq \{1, 2, \dots, 2n\}$ is called {\em admissible} if 
\begin{equation}
\label{eq:admi}
 |I \cap \{1, 2, \dots, m\}| \leq \frac{m}{2},
 \end{equation}
for all $m=1, \dots, 2k$. We let 
$$A^n_k = \{I \subseteq \{1, 2, \dots, 2n\} \mid I \text{ is admissible, } |I| = k\}.$$ 
\end{definition}

\begin{proposition}[Lemma 8 in \cite{KhSpringer}]
\label{prop:adm}
A basis for $Z(H^n)_{2k}$ consists of the elements $X_I$ for $I \in A^n_k$.
\end{proposition}

\begin{example}
When $n=2$, the ranks of $Z(H^2)$ in degrees $0, 2$ and $4$ are $1$, $3$ and $2$, respectively. A  basis is given by $1$, $X_2$, $X_3$, $X_4$, $X_2X_4$ and $X_3X_4$.
\end{example}

\subsection{The dual of the center} 
\label{sec:dual}
For a free Abelian group $V$, we will denote by 
$V^\vee = \Hom(V, \Z)$ its dual. For example, $Z(H^n)^\vee_{2k}$ is the dual of $Z(H^n)_{2k}$. In view of \eqref{eq:choose}, this is a free Abelian group of rank ${2n \choose k} - {2n \choose k-1}$. We will describe a set of interesting elements in $Z(H^n)_{2k}$. 

Let $Z^n_k$ be the Abelian group freely generated by elements $X_I$ for $I \subseteq \{1, \dots, 2n\}$ with $|I|=k$. The dual $(Z^n_k)^\vee$ has a dual basis given by $X_I^\vee$, where $ X_I^\vee(X_J)$ is Kronecker's $\delta_{IJ}$. We introduce a multiplication on $\oplus_k (Z^n_k)^\vee$ by setting
$$ X_I^{\vee} \cdot X_J^{\vee} =  \begin{cases} X_{I \cup J}^\vee& \text{if } I \cap J =\emptyset,\\
0 & \text{otherwise.}
\end{cases}$$
We can define a contraction operation $(Z^n_{\ell})^\vee\times Z^n_k\to  Z^n_{k-\ell}$ by
$$ X_I^{\vee} (X_J) =  \begin{cases} X_{J\backslash I}& \text{if } I \subseteq J,\\
0 & \text{otherwise.}
\end{cases}$$
When $k=\ell$, this agrees with the action of the dual group. Furthermore, we have the associativity relation
\begin{equation}\label{eq:contraction} (X_I^{\vee} \cdot X_J^{\vee})(X_K)=X_I^{\vee} (X_J^{\vee}((X_K))\end{equation}
Recall that $Z(H^n)_{2k}$ is the quotient of $Z^n_k$ by the relations  \eqref{eq:ZHnj}. It follows that 
$Z(H^n)^\vee_{2k}$ is a subgroup of $(Z^n_k)^\vee$ consisting of those functions $f: Z^n_k\to \Z$ such that
\begin{equation}
\label{eq:ZHnjd}
f\bigl( \sum_{|I | = k, \ I \supseteq J} X_I \bigr) = 0,
\end{equation}
for every subset $J \subseteq \{1, \dots, 2n\}$ with $|J| < k$.

\begin{definition}
A {\em partial matching} of $\{1, 2, \dots, 2n\}$ is a set
$$ \mm = \{p_1,\dots, p_k\}$$
consisting of $k$ disjoint pairs of elements from $\{1, 2, \dots, 2n\}$, for some $k \leq m$. A pair $p = (i, j)$ is called {\em balanced} if it consists of an odd number and an even number, and a partial matching $\mm$ is called {\em balanced} if all the pairs $p_i$ in $\mm$ are balanced.
\end{definition}

\begin{example}
The following is a balanced partial matching of $\{1, \dots, 8\}$:
$$ \{ (2, 5), (3, 8), (6, 7) \}.$$
\end{example}

Given a partial matching $\mm=\{(i_1, j_1), \dots, (i_k, j_k)\}$, we define an element $f_{\mm} \in (Z^n_k)^\vee$ by
$$ f_{\mm} = \prod_{s=1}^k (X_{i_s}^\vee - X_{j_s}^\vee).$$

\begin{lemma}
\label{lem:match}
For every partial matching $\mm$, the element $f_{\mm}$ satisfies the relations \eqref{eq:ZHnjd}, and therefore can be viewed as an element of $Z(H^n)^\vee_{2k}$.
\end{lemma}

\begin{proof}
We show that the expression
\begin{equation}
\label{eq:expand}
\Bigl( \prod_{s=1}^k (X_{i_s}^\vee - X_{j_s}^\vee) \Bigr) \bigl( \sum_{|I | = k, \ I \supseteq J} X_I \bigr)
\end{equation}
vanishes. First, we observe that all terms in \eqref{eq:expand} are zero if $J$ contains elements that do not appear in $\mm$. Indeed, each term in the sum is of the form $$\pm \left(\prod_{s=1}^k X_{\nu_s}^\vee\right) X_I,$$
where the symbol $\nu$ is either $\nu=i$ or $\nu=j$. By the definition of the product structure on the dual group, this term is only non-zero if $I=\{\nu_1,\dots,\nu_k\}$. However, since $I\supseteq J$, we must have that $J$ is contained in the set of elements in $\mm$.

We now assume that all elements of $J$ appear in $\mm$. Since $|J| <k$, there is at least one pair in $\mm$, say $(i_1, j_1)$, that does not contain elements of $J$. We split up the sum \eqref{eq:expand} according to whether $I$ contains $i_1,j_1,$ or both. If $I$ contains neither $i_1$ nor $j_1$, then it contributes zero after applying the dual element $X_{i_1}^\vee-X_{j_1}^\vee$. Applying \eqref{eq:contraction} to contract by $X_{i_1}^\vee-X_{j_1}^\vee$, \eqref{eq:expand} becomes
\begin{align*}
\label{eq:expand2}
& \Bigl( \prod_{s=2}^k (X_{i_s}^\vee - X_{j_s}^\vee) \Bigr)(X_{i_1}^\vee-X_{j_1}^\vee)\bigl( \sum_{\substack{|I | = k\\ \ I \supseteq J\cup\{i_1\} \\ j_1\not\in I}} X_I + \sum_{\substack{|I | = k \\ \ I \supseteq J\cup\{j_1\}\\ i_1\not\in I}} X_I + \sum_{\substack{|I | = k \\ \ I \supseteq J\cup\{i_1,j_1\}}} X_I  \bigr) \\
= & \Bigl( \prod_{s=2}^k (X_{i_s}^\vee - X_{j_s}^\vee) \Bigr)\bigl( \sum_{\substack{|I | = k-1 \\ \ I \supseteq J \\ i_1,j_1\not\in I}} X_I - \sum_{\substack{|I | = k-1 \\ \ I \supseteq J\\ j_1,i_1\not\in I}} X_I + \sum_{\substack{|I | = k-1 \\ \ I \supseteq J \cup \{j_1\}\\ i_1\not\in I}} X_I  - \sum_{\substack{|I| = k-1 \\ \ I \supseteq J \cup \{i_1\}\\ j_1\not\in I}} X_I \bigr) \\
\end{align*}
The first two sums inside the last parentheses manifestly cancel. We can form a bijection between the terms in the third and fourth sums as follows: if $I$ is a subset contributing to the third term, then $I'=\{j_1\}\cup I\backslash\{i_1\} $ is a subset contributing to the fourth sum. Since neither $i_1$ nor $j_1$ appear among the $i_s,j_s$ for $s>1$, we see that $$\Bigl( \prod_{s=2}^k (X_{i_s}^\vee - X_{j_s}^\vee) \Bigr) X_I = \Bigl( \prod_{s=2}^k (X_{i_s}^\vee - X_{j_s}^\vee) \Bigr) X_{I'}$$
for such an $I$. Therefore, the terms in the third and fourth sums cancel as well.
\end{proof}

We now exhibit a set of generators for $Z(H^n)^\vee_{2k}$. (Note that it will usually not be a basis.)

\begin{proposition}
\label{prop:balancedmatch}
The elements $f_{\mm}$, over all balanced partial matchings $\mm$ of cardinality $k$, generate the group  $Z(H^n)^\vee_{2k}$.
\end{proposition}

\begin{proof}
Let us first show that $f_{\mm}$, over all (not necessarily balanced) partial matchings $\mm$ of cardinality $k$, generate $Z(H^n)^\vee_{2k}$. Let $V$ be the Abelian group freely generated by  partial matchings of cardinality $k$. We need to show that the linear homomorphism
$$ V \to Z(H^n)^\vee_{2k}, \ \ \mm \mapsto f_{\mm}$$
is surjective. This is equivalent to showing that its dual $Z(H^n)_{2k} \to V^\vee$ is injective. Proposition~\ref{prop:adm} tells us that a basis of $Z(H^n)_{2k} $ is given by $X_I$ with $I \in A^n_k$. Therefore, what we need to check is that, if we have numbers $a_I \in \Z$ such that
\begin{equation}
\label{eq:ai}
 f_{\mm} \bigl(\sum_{I \in A^n_k} a_I X_I \bigr) = 0,
 \end{equation}
for all partial matchings $\mm$ of cardinality $k$, then $a_I = 0$ for all $I \in A^n_k$.

We will prove this claim by induction on $n$. The base case $n=0$ is clear, because $f_\emptyset = 1$. 

For the inductive step, assume the corresponding statement is true for $n-1$ and all possible $k$. Suppose we have numbers $a_I$ satisfying \eqref{eq:ai}. Consider first the partial matchings $\mm$ that consist only of pairs not involving the last two elements $2n-1$ and $2n$. For such $\mm$, we have $f_{\mm}(X_I) = 0$ when $I \cap \{2n-1, 2n\} \neq \emptyset$. If $I \in A^n_k$ has $I \cap \{2n-1, 2n\} = \emptyset$, then $I$ is an admissible subset of $\{1, \dots, 2n-2\}$, and we can also view $\mm$ as a partial matching of $\{1, \dots, 2n-2\}$. Applying the inductive hypothesis for $n-1$ and $k$, we deduce that
 \begin{equation}
\label{eq:ai0}
a_I = 0, \ \forall I \in A^n_k \text{ with } I \cap \{2n-1, 2n\} = \emptyset.
 \end{equation}

Next, consider an arbitrary partial matching $\mm$ of $\{1, \dots, 2n-2\}$ of cardinality $k-1$. If $k < n$, there exists some $i \in \{1, \dots, 2n-2\}$ that does not appear in any of the pairs in $\mm$. Define the matching
$$\mm' = \mm \cup \{(2n-1, i)\}$$
so that 
$$ f_{\mm'} = (X_{2n-1}^\vee - X_i^\vee) \cdot f_{\mm}.$$
Applying \eqref{eq:ai} for $\mm'$, and using \eqref{eq:ai0}, we get 
$$ 0= f_{\mm'} (\sum_{I \in A^n_k} a_I X_I) =  f_{\mm} \Bigl(\sum_{\substack{I = J \cup \{2n-1\} \\ \ J \in A^{n-1}_{k-1}}} a_I X_J \Bigr).$$
Since this is true for all possible $\mm$, from the inductive hypothesis for $n-1$ and $k-1$, we deduce that
 \begin{equation}
 \label{eq:ai1}
a_I = 0, \ \forall I\in A^n_k \text{ with } I \cap \{2n-1, 2n\} = \{2n-1\}.
 \end{equation}
 Observe that if $I \cap \{2n-1, 2n\} = \{2n-1\}$, the admissibility condition \eqref{eq:admi} for $I$ applied to $m=2n-1$ shows that our hypothesis $k < n$ must be satisfied. 
 
Let $\mm$ still be a partial matching of $\{1, \dots, 2n-2\}$ of cardinality $k-1$, and set
$$ \mm'' =\mm \cup \{(2n, 2n-1)\}$$
so that
$$f_{\mm''} = (X_{2n}^\vee - X_{2n-1}^\vee) \cdot f_{\mm}.$$
 Applying \eqref{eq:ai} for $\mm'$, and using \eqref{eq:ai0} and \eqref{eq:ai1}, we find that
 $$  0= f_{\mm''} (\sum_{I \in A^n_k} a_I X_I) =  f_{\mm} \Bigl(\sum_{\substack{I = J \cup \{2n\}\\J \in A^{n-1}_{k-1}}} a_I X_J \Bigr).$$
 Applying the inductive hypothesis for $n-1$ and  $k-1$, we deduce that
  \begin{equation}
 \label{eq:ai2}
a_I = 0, \ \forall I\in A^n_k \text{ with } I \cap \{2n-1, 2n\} = \{2n\}.
 \end{equation}
 
 Finally, consider an arbitrary partial matching $\mm$ of $\{1, \dots, 2n-2\}$ of cardinality $k-2$. Since $k \leq n$, we can find $i, j \in \{1, \dots, 2n-2\}$ that are not in any pair in $\mm$. Let
 $$ \mm' = \mm \cup \{(2n-1, i), (2n, j)\}$$
so that 
$$ f_{\mm'} = (X_{2n-1}^\vee - X_i^\vee) ( X_{2n}^\vee - X_j^\vee)\cdot f_{\mm}.$$
Applying \eqref{eq:ai} for $\mm'$, and using \eqref{eq:ai0}, \eqref{eq:ai1} and \eqref{eq:ai2}, we get
  $$  0= f_{\mm'} (\sum_{I \in A^n_k} a_I X_I) =  f_{\mm} \Bigl(\sum_{\substack{I = J \cup \{2n-1, 2n\}\\ J \in A^{n-1}_{k-2} }} a_I X_J \Bigr).$$
 From the inductive hypothesis for $n-1$ and  $k-2$, we conclude that
   \begin{equation}
 \label{eq:ai3}
a_I = 0, \ \forall I\in A^n_k \text{ with }  \{2n-1, 2n\} \subseteq I.
 \end{equation}
 
This shows that all $a_I$ vanish, and therefore $f_{\mm}$ generate $Z(H^n)^\vee_{2k}$. 

To see that $f_{\mm}$ for balanced $\mm$ also suffice to generate $Z(H^n)^\vee_{2k}$, we will prove that every $f_{\mm}$ is a linear combination of the balanced ones. We will do this inductively: If $\mm$ is not balanced, we will express $f_{\mm}$ as a linear combination of elements corresponding to matchings that have fewer unbalanced pairs than $\mm$.

An unbalanced partial matching $\mm$ must contain a pair of odd elements, or a pair of even elements. Suppose it contains both: a pair $(a, b)$ of odd elements, and a pair $(c, d)$ of even elements. Using the relation
$$ (X_a^\vee - X_b^\vee) (X_c^\vee- X_d^\vee) = (X_a^\vee - X_c^\vee) (X_b^\vee- X_d^\vee) +(X_a^\vee - X_d^\vee) (X_c^\vee- X_b^\vee)$$
we can turn $f_{\mm}$ into a sum $f_{\mm'} + f_{\mm''}$, where $\mm'$ and $\mm''$ have two fewer unbalanced pairs than $\mm$.

If $\mm$ does not contain both types of unbalanced pairs, without loss of generality let us suppose it only contains pairs made of even elements, in addition to possibly some balanced pairs. In total, there are more even than odd elements in the pairs in $\mm$, so there must be an odd number $a \in \{1, \dots, 2n\}$ that is not contained in any pair in $\mm$. Let $(b, c)$ be a pair in $\mm$ with $b$ and $c$ both even. Using the relation
$$ X_b^\vee- X_c^\vee = (X_b^\vee- X_a^\vee) + (X_a^\vee- X_c^\vee)$$
we can express $f_{\mm}$ as $f_{\mm'} + f_{\mm''}$, where $\mm'$ and $\mm''$ have one fewer unbalanced pair compared to $\mm$. This completes the proof.
\end{proof}

\subsection{Hochschild homology and cohomology}\label{sec:hochschild}
To study the cabled Khovanov-Rozansky homology of the $p$-framed unknot, we need to first understand the homologies of the cables of $(U, p)$. Since we restrict ourselves to level $\alpha=0$, these cables are the $(2n,2np)$-torus links $T'_{2n,2np}$, with $n$ strands positively oriented and $n$ strands negatively oriented, going through $p$ full twists. The Khovanov homology of these links was studied by Sto\v{s}i\'{c} in \cite{stosic}. There, for $p > 0$, he showed that
$$\Kh^{i, j}(T'_{2n, 2np}) = 0  \ \ \text{if} \ i > 0 \text{ or } j > 0$$
and, furthermore, in the maximal homological degree $0$ we have
$$ \Kh^{0, j}(T'_{2n, 2np}) = \begin{cases}
\Z^{{2n \choose n-k} - {2n \choose n-k-1}} & \text{if } j=-2k, \ k=0, \dots, n,\\
0 & \text{otherwise.}
\end{cases}$$
See \cite[Corollaries 2 and 4]{stosic}. For $p < 0$, the formula \eqref{eq:Ext} for the Khovanov homology of the mirror gives
$$ \Kh^{0, j}(T'_{2n, 2np}) = \begin{cases}
\Z^{{2n \choose n-k} - {2n \choose n-k-1}} & \text{if } j=2k, \ k=0, \dots, n,\\
0 & \text{otherwise.}
\end{cases}$$

Note the similarity between these answers and \eqref{eq:choose}. (When $n=1$, this was first observed by Przytycki in \cite{Prz}.) In fact, a more direct relation between $ \Kh^{0, j}(T'_{2n, 2np})$ and $Z(H^n)$ comes from \cite{Roz}. There, Rozansky constructs a Khovanov homology for framed links in $S^1 \times S^2$. We will only need the case of null-homologous links in $S^1 \times S^2$, in which case the framing dependence can be cancelled by a suitable shift in gradings, as shown by Willis in \cite{Willis}. The Khovanov homology of a null-homologous link $L \subset S^1 \times S^2$ is a well-defined bi-graded group, defined as follows.

Suppose that $L$ is given as the circular closure of a tangle $T$ from $2n$ to $2n$ points. In the standard  picture of $S^1 \times S^2$ as $0$-surgery on the unknot, this corresponds to connecting the $2n$ pairs of points by arcs going through the unknot:
\[ 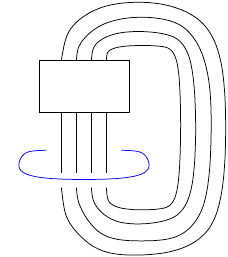
\]

Recall from \eqref{eq:FT} that to the tangle $T$, Khovanov associated a complex of $(H^n, H^n)$-bimodules $\F(T)$. Rozansky's extension of Khovanov homology to links in $S^2\times S^1$ sets the invariant of $L$ to be the Hochschild homology of $\F(T)$:
\begin{equation}
\label{eq:KhHH}
 \Kh^{i, *}(S^2\times S^1; L) := \HH_i(\F(T)).
 \end{equation}
Moreover, in \cite[Theorem 6.8]{Roz}, Rozansky shows that there is a canonical isomorphism
\begin{equation}
\label{eq:introducetwists}
\Kh^{i, *}(S^2\times S^1; L) \cong \Kh^{i, *}(L(p)) \ \ \ \text{for } \ i \geq n_+ - 2p+2,
\end{equation} where $n_+$ is the number of positive crossings in a diagram for $T$, and $L(p) \subset S^3$ is the link obtained by inserting $p$ full twists in the corresponding diagram for $L$ at the place where the $2n$ arcs went through the $0$-framed unknot:
\[ 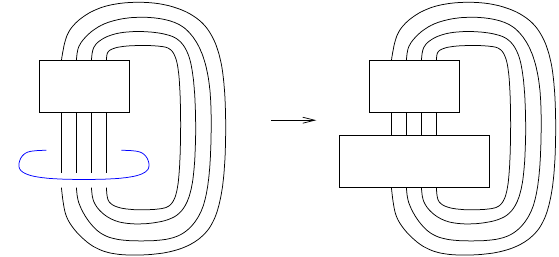
\]
This isomorphism is canonical in the sense that it arises from a natural isomorphism of certain functors. Let $\operatorname{FT}_{2n}$ be the full twist tangle on $2n$ strands, so that $L(p)$ is the circular closure of $(\operatorname{FT}_{2n})^p\circ T$. According to \cite[Theorem 6.7]{Roz}, the complex $\F((\operatorname{FT}_{2n})^p)$ is a projective resolution of the identity bimodule $H^n$ in a prescribed range of degrees. Let $\mathcal{P}(-)$ be any functorial projective resolution. Then, when $i \geq n_+ - 2p+2$, the functors
$$T\to \HH_i(\F(T)) := H_i\left(\F(T)\bigotimes_{H^n\otimes (H^n)^{\operatorname{op}}}\mathcal{P}(H^n)\right)$$
and
$$T\to \Kh^{i, *}(L(p)) = H_i\left(\F(T)\bigotimes_{H^n\otimes (H^n)^{\operatorname{op}}}\F((\operatorname{FT}_{2n})^p)\right)$$
compute the same $\operatorname{Tor}$ group. In particular, the canonical isomorphism in Equation \eqref{eq:introducetwists} is natural with respect to maps induced from cobordisms between tangles. 

Let us specialize to the case when $i=0$ and $T$ is the identity tangle $\Id_{n, n}$ on $2n$ alternately oriented strands. As in \cite{MMSW}, we denote the corresponding link $L \subset S^1 \times S^2$ by $F_{n,n}$. We have 
$$F_{n,n}(p) = T'_{2n, 2np}.$$
We find that, for every $p > 0$, 
$$ \Kh^{0, *}(T'_{2n, 2np}) \cong \HH_0(H^n).$$
Using \eqref{eq:Ext}, we get a description of $\Kh^{0, *}(T'_{2n, 2np})$ for $p <0$. By the chain of isomorphisms following Equation 6.13 in \cite{Roz} with $H^n$ replaced by $\F(T)$, we have a duality isomorphism
\begin{equation}\label{eq:hochschildduality}
\HH^i(\F(T))\cong\HH_i(\F(\overline{T}))^{\vee}\{2n\},
\end{equation}
for an $(n,n)$-tangle $T$ with mirror $\overline{T}$. In particular, taking $T$ to be the identity tangle gives an isomorphism (up to a degree shift) between the Hochschild homology and cohomology of $H^n$. Applying this isomorphism for $i=0$, $$ \Kh^{0, *}(T'_{2n, 2np}) \cong \Kh^{0, *}(T'_{2n, -2np})^\vee \cong \HH_0(H^n)^{\vee} \cong \HH^0(H^n)\{-2n\} \ \text{ for } p < 0.$$ See \cite[Theorem 6.9]{Roz}. (Some care has to be taken with respect to grading conventions: Rozansky puts $X$ in degree $2$, so its quantum grading is the negative of the usual quantum grading in $\Kh$.) 

The zeroth Hochschild cohomology of a ring equals its center. Therefore, we have canonical isomorphisms
\begin{equation}
\label{eq:Khp+}
 \Kh^{0, j}(T'_{2n, 2np}) \cong Z(H^n)^{\vee}_{2n+j} \ \text{ for } \ p > 0
    \end{equation}
and
\begin{equation}
\label{eq:Khp-}
  \Kh^{0, j}(T'_{2n, 2np}) \cong Z(H^n)_{2n-j} \ \text{ for } \ p < 0.
  \end{equation}

Let us re-write \eqref{eq:Khp+} and \eqref{eq:Khp-} in terms of the Khovanov-Rozansky homology $\KhR_2$, which is related to $\Kh$ by the formula \eqref{eq:KhR2}. We are interested in $T'_{2n, 2np}$ as a framed link, in which every component has framing $p$. A diagram for this framed link is obtained from the standard diagram of the torus link $T'_{2n, 2np}$ (which has writhe $-2np$) by adding $p$ kinks in each component. The writhe of the resulting diagram is then $w=0$. Therefore, in view of \eqref{eq:KhR2}, we obtain:
\begin{equation}
\label{eq:KhRp+}
  \KhR_2^{0, j}(T'_{2n, 2np}) \cong Z(H^n)_{2n+j} \ \text{ for } \ p > 0
  \end{equation}
and
\begin{equation}
\label{eq:KhRp-}
 \KhR_2^{0, j}(T'_{2n, 2np}) \cong Z(H^n)^{\vee}_{2n-j} \ \text{ for } \ p < 0.
   \end{equation}
Note that Theorem~\ref{thm:ZHn} gives an explicit description of the center $Z(H^n)$. 

\subsection{Cobordism maps and the braid action}\label{sec:cobordismmaps}
In the definition of the cabled Khovanov-Rozansky homology we had the cobordism maps $\psi_i^{[m]}$. In our case, there is a single knot component, so we will drop the subscript $i=1$ from the notation. Further, since $N=2$, the values of $m$ can be $0$ or $1$. We will simply write $\psi$ for $\psi^{[0]}$ and $\phi$ for $\psi^{[1]}.$

Thus, we are interested in the maps:
\begin{equation}
\label{eq:psi}
\psi = \KhR_2(Z)( \cdot \otimes 1):  \KhR_2^{0, j}(T'_{2n, 2np}) \to  \KhR_2^{0, j}(T'_{2n+2, (2n+2)p})
 \end{equation}
and
\begin{equation}
\label{eq:phi}
\phi = \KhR_2(Z)( \cdot \otimes X):\KhR_2^{0, j}(T'_{2n, 2np}) \to  \KhR_2^{0, i+2}(T'_{2n+2, (2n+2)p}),
 \end{equation}
where $Z=Z_1$ is the saddle cobordism from $T'_{2n, 2np}\sqcup U$ to $T'_{2n+2, (2n+2)p}$, which introduces two new strands in the cable; cf. Section~\ref{sec:cabled}. We are interested in computing $\psi$ and $\phi$ as maps relating $Z(H^n)$ and $Z(H^{n+1})$, under the identifications \eqref{eq:KhRp+} and \eqref{eq:KhRp-}.

For this, we introduce the $(n+1, n+1)$-tangle
\[ 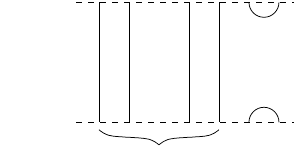
\]
We denote by $M=\F(\tangle)$ the $(H^{n+1}, H^{n+1})$-bimodule associated to $\tangle$. 

Observe that the circular closure of $\tangle$ in $S^1 \times S^2$ is the following link:
\[ \begin{picture}(0,0)%
\includegraphics{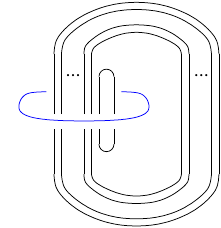}%
\end{picture}%
\setlength{\unitlength}{3158sp}%
\begingroup\makeatletter\ifx\SetFigFont\undefined%
\gdef\SetFigFont#1#2#3#4#5{%
  \reset@font\fontsize{#1}{#2pt}%
  \fontfamily{#3}\fontseries{#4}\fontshape{#5}%
  \selectfont}%
\fi\endgroup%
\begin{picture}(2202,2260)(1336,-2616)
\put(1351,-1561){\makebox(0,0)[lb]{\smash{{\SetFigFont{10}{12.0}{\rmdefault}{\mddefault}{\updefault}{\color[rgb]{0,0,1}$0$}%
}}}}
\end{picture}%

\]
This is the split disjoint union $F_{n, n} \sqcup U$, which can also be represented as the circular closure of the $(n, n)$-tangle $\Id_{n, n} \sqcup \; U$. Therefore, we have two different ways of expressing the Khovanov homology of $F_{n, n} \sqcup U$ in terms of Hochschild homology:
\begin{equation}
\label{eq:twoHH}
\HH_i(M) \cong  \Kh^{i, *}(S^2\times S^1; F_{n, n}) \cong \Kh^{i,*}(S^2\times S^1; F_{n, n}) \otimes \A \cong \HH_i(H^n) \otimes \A.
\end{equation}

To be in line with the conventions in this paper, we will work with $\KhR_2$ instead of $\Kh$; compare \eqref{eq:KhR2}. Thus, from Equation~\eqref{eq:twoHH} for $i=0$, we have 
\begin{align*}
 \HH_0(H^n)^{\vee} \otimes \A^{\vee} & \cong \HH_0(M)^{\vee},\\
\HH^0(H^n)\{-2n\}\otimes \A\{-2\} & \cong \HH^0(M)\{-2n-2\},
 \end{align*}
 where we use Equation \eqref{eq:hochschildduality} to get the second line. We thereby obtain an isomorphism
 \begin{equation}
\label{eq:Xi}
 \HH^0(H^n)\otimes \A \xrightarrow{ \phantom{a} \cong \phantom{a}} \HH^0(M).
 \end{equation}
 We have $\HH^0(H^n) = Z(H^n)$, whereas $\HH^0(M)$ is given by
\begin{equation}
\label{eq:hm}
 \HH^0(M) =\{ m \in M \mid mh = hm, \ \forall h \in H^{n+1}\}.
 \end{equation}
 
 Proposition~\ref{prop:HM} below will give an explicit formula for the isomorphism in \eqref{eq:Xi}, up to sign. Before stating it, let us introduce some notation for certain elements of the bimodule $M$. Recall that 
 $$M = \bigoplus_{a, b \in \Cr_{n+1}}  \M{a}{b},$$
 where $\M{a}{b}$ is the complex associated to the link $\bar b J a$.
 
Suppose we have a crossingless matching $a \in \Cr_n$. From $a$ we can construct crossingless matchings in $\Cr_{n+1}$ in two ways. First, we get a matching ${\ta} \in \Cr_{n+1}$ by connecting the last two endpoints $p_{2n+1}$ and $p_{2n+2}$. Then, the link $\overline{\ta} J {\ta}$ is the split disjoint union of $\bar a a$ and two unknots:
\[ 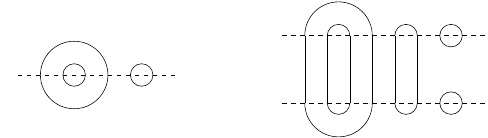
\]
Given an element $x \in \Hn{a}{a}$ and $w \in \KhR_2(U \sqcup U) \cong \A \otimes \A$, we get an element 
$$ x \otimes w  \in \M{\ta}{\ta} \subset M.$$

Second, given $a \in \Cr_n$, let $\Out(a)$ denote the set of ``outer'' arcs in $a$, that is, those connecting points $p_i$ and $p_j$ (for $i < j$) such that no points $p_k$ and $p_l$ with $k < i < j < l$ are matched in $a$. For example, when $a$ is the matching
\[ 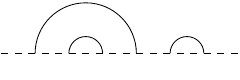
\]
the outer arcs are those from $p_1$ to $p_4$, and from $p_5$ to $p_6$. For $e \in \Out(a)$ connecting $p_i$ to $p_j$ with $i < j$, we define a crossingless matching $a_{\to}^e \in \Cr_{n+1}$ by connecting $p_i$ to $p_{2n+2}$ and $p_j$ to $p_{2n+1}$. Notice that the link $\overline {a_{\to}^e} J a_{\to}^e$ is diffeomorphic to $\overline{a}a$:
\[ 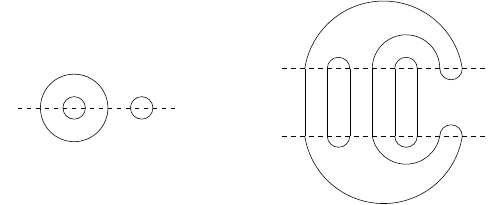
\]
Under this diffeomorphism, an element $x \in \Hn{a}{a}$ produces a corresponding element
$$ x_{\to}^e  \in \M{a_\to^e}{a_\to^e} \subset M$$
 We let
 $$ x_{\to} := \sum_{e \in \Out(a)} x_{\to}^e \ \in M.$$
 The assignment $x \mapsto x_{\to}$ extends linearly to a map
 $$ \bigoplus_{a \in {\Cr_n}} \Hn{a}{a} \to \bigoplus_{a \in \Cr_{n+1}} \M{a}{a}, \ \ \ x \mapsto x_{\to}$$
 Furthermore, given $v \in \A$, we will denote by $x_{\to} \cdot v$ the result of acting by $v$ using the multiplication coming from a small unknot near $p_{2n}$. In other words, in terms of the module action described at the end of Section~\ref{sec:Hn}, we identify the variable $X \in \A$ with $X_{2n}$ and act accordingly.
 
 Let us define a map
 $$ \Theta: Z(H^n) \otimes \A \to M, \ \ \ \Theta(x \otimes v) = x \otimes \Delta(v) + x_{\to} \cdot v.$$
 Observe that the map $\Theta$ is injective (because so is $\Delta$).

\begin{proposition}
\label{prop:HM}
The image of the map $\Theta$ is $\HH^0(M) \subset M$, and the isomorphism from \eqref{eq:Xi} is given by $\pm \Theta$.
\end{proposition}

\begin{proof}
We will determine the structure of elements of $\HH^0(M)$ in order to identify it as the image of $\Theta$. Given $m \in \HH^0(M)$, let us write
$$ m = \sum_{a, b \in \Cr_{n+1}} \! \! \! \! \m{a}{b} \ , \ \  \m{a}{b} \in \M{a}{b}. $$
The defining property of elements $m \in \HH^0(M)$ is that they commute with all elements of $H^{n+1}$; cf. \eqref{eq:hm}. In particular, they commute with the idempotents ${}_a1_a$ corresponding to each $a \in \Cr_{n+1}$. It follows that
$$  \m{a}{b} = 0 \ \ \text{if} \ a\neq b.$$

Note that every crossingless matching $b \in \Cr_{n+1}$ is either of the form $\ta$ or $a_{\to}^e$, for some $a \in \Cr_n$ and $e \in \Out(a)$. Let us write
\begin{equation}
\label{eq:mmm}
 m=m' + m''
 \end{equation}
where
\begin{equation}
\label{eq:mpp}
 m' = \sum_{ a \in \Cr_n} \m{\ta}{\ta} \ \ \ \text{and} \ \ \ m'' =  \sum_{a \in \Cr_n} \sum_{e \in \Out(a)} \m{a_{\to}^e}{a_{\to}^e}.
\end{equation}

Let us analyze the commutation relations between $m$ and elements 
$$ h \in \bigoplus_{ a \in \Cr_n} \Hnp{\ta}{\ta} \cong  \bigoplus_{ a \in \Cr_n}  \Hn{a}{a} \otimes \A \cong H^n \otimes \A.$$
We have $m''h=hm''=0$, and therefore $m'h =hm'.$ Notice that  
$$m' \in \bigoplus_{ a \in \Cr_n} \M{\ta}{\ta} \cong \bigoplus_{ a \in \Cr_n}\Hn{a}{a} \otimes \A^{\otimes 2} \cong H^n \otimes \A^{\otimes 2}.$$
Letting $m'=\sum\limits_{0\leq i,j\leq 1} w_{i,j}\otimes X^i\otimes X^j$ and $h=v\otimes X^k$ for $w_{i,j,},v\in H^n$ and $k=0,1$, the relation $m'h = hm'$ implies that
\begin{equation}\label{eq:commutation}
\sum\limits_{i,j} (w_{i,j}v)\otimes X^i \otimes X^{j+k} = \sum\limits_{i,j} (vw_{i,j})\otimes X^{i+k}\otimes X^j
\end{equation}
Taking $k=0$, we find that $w_{i,j}v=vw_{i,j}$ for each $i,j$. Since this holds for all $v$, we have $w_{i,j}\in Z(H^n)$. Taking $v=1$ and $k=1$, \eqref{eq:commutation} gives that $w_{00}=0$ and $w_{01}=w_{10}$. Thus, $m'=w\otimes (1\otimes X+X\otimes 1)+w'\otimes X\otimes X$, with $w,w'\in Z(H^n)$. Note that the span of $1\otimes X+X\otimes 1$ and $X\otimes X$ is precisely the image of the comultiplication $\Delta$. Therefore,
\begin{equation}
\label{eq:mp}
 m' \in Z(H^n) \otimes \Delta(\A).
 \end{equation}

Next, let $a \in \Cr_n$ and $e \in \Out(a)$. Consider the element $h \in \Hn{\ta}{a_{\to}^e}$ obtained by marking with $1$ all the circles in the tangle $\overline{a_{\to}^e} \ta$. The commutation $mh=hm$ reduces to the relation:
\begin{equation}
\label{eq:mah}
 (\m{\ta}{\ta} ) \cdot h = h \cdot (\m{a_{\to}^e}{a_{\to}^e}).
 \end{equation}
 Observe that the left hand side of \eqref{eq:mah} is given by multiplying the values from two circles in $\overline{a_{\to}^e} J \ta$ (the one containing the arc $e$ and the one at the top through $p_{2n+1}$ and $p_{2n+2}$), while keeping the values on the other circles of $\overline{a_{\to}^e} J \ta$ the same, as in the following example:
\[ 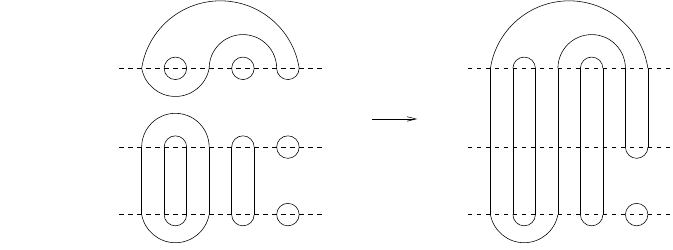
\]
 
On the other hand, the right hand side of \eqref{eq:mah} is given by the comultiplication $\Delta$ applied to the copy of $\A$ coming from the circle going through the last two points $p_{2n+1}$ and $p_{2n+2}$:
 \[ 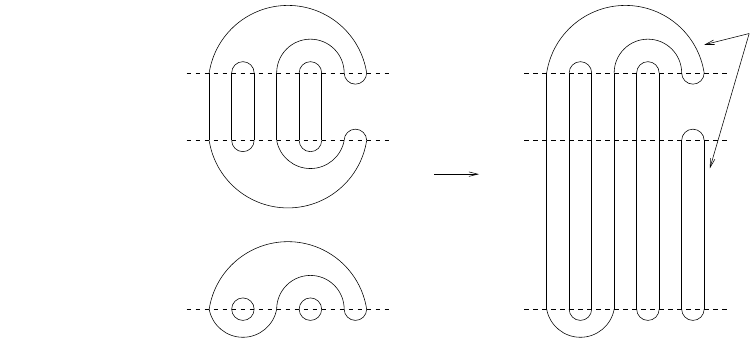
\]

Therefore, if we write
$$\m{\ta}{\ta} = \sum_i x_i \otimes \Delta(v_i)  \in \Hn{a}{a} \otimes \Delta(\A),$$
from \eqref{eq:mah} we deduce that
\begin{equation}
\label{eq:mae}
 \m{a_{\to}^e}{a_{\to}^e}= \sum_i (x_i)_{\to}^e \cdot v_i.
 \end{equation}
It follows from \eqref{eq:mmm}, \eqref{eq:mpp}, \eqref{eq:mp} and \eqref{eq:mae} that $m$ is in the image of $\Theta$. Conversely, it can be checked that all the elements in the image of $\Theta$ commute with every $h \in H^{n+1}$. Therefore,
$$ \HH^0(M) = \operatorname{Im}(\Theta) \subset M.$$

To pin down the isomorphism \eqref{eq:Xi}, we use the module action by $R^{\otimes (2n+1)}$. Recall from Proposition~\ref{prop:ZHn} the description of $\HH^0(H^n) =Z(H^n)$ as an $R^{\otimes 2n}$-algebra. It follows that 
$$ Z(H^n) \otimes \A \cong R^{\otimes(2n+1)}/\bigl \langle \sum_{|I| =k} X_I, \ \ k=1, \dots, 2n\bigr \rangle.$$
On the other hand, as noted at the end of Section~\ref{sec:Hn}, an $H^{n+1}$-module such as $M$ admits an action of $R^{2n+2}$ (say, from the points at the bottom of the tangle). Since the last two points are connected in the tangle $J$, it follows that $X_{2n+2}$ acts on $M$ by $-X_{2n+1}$. We can thus focus on the action of $R^{\otimes (2n+1)}$ on $M$, using the first $2n+1$ variables. This descends to an action of $R^{\otimes (2n+1)}$ on $\HH^0(M) \subset M$. The constructions in \cite{Roz} preserve the $R^{\otimes (2n+1)}$ actions, and therefore the isomorphism \eqref{eq:Xi} is one of (graded) $R^{\otimes (2n+1)}$-modules. 

Observe also that the map $\Theta$ preserves the module actions. A graded automorphism of 
\begin{equation}
\label{eq:Ro}
R^{\otimes(2n+1)}/\bigl \langle \sum_{|I| =k} X_I, \ \ k=1, \dots, 2n\bigr \rangle
\end{equation}
as an $R^{\otimes (2n+1)}$-module is determined by the image of $1^{\otimes(2n+1)}$. The only elements in $R^{\otimes(2n+1)}$ with the same grading as $1^{\otimes(2n+1)}$ are elements of the form $m_1\otimes\dots\otimes m_{2n+1}$ for $m_i\in\Z\subset R$. In order for this to be an automorphism, we must have that all of the $m_i$ are $\pm 1$. That is, the only graded automorphisms of \eqref{eq:Ro} are $\pm \Id$. It follows that given two modules isomorphic to \eqref{eq:Ro}, the isomorphism between them is uniquely determined up to sign. Therefore, \eqref{eq:Xi} must be given by $\pm \Theta$. (We conjecture that it is $\Theta$.) 
\end{proof}

We can now compute the maps from \eqref{eq:psi} and \eqref{eq:phi}.

\begin{proposition}
\label{prop:plus}
Let $p > 0$. Under the identification \eqref{eq:KhRp+}, and using the description of $Z(H^n)$ from Theorem~\ref{thm:ZHn}, the maps $\psi$ and $\phi$ from \eqref{eq:psi}, \eqref{eq:phi} are given by
\begin{equation}
\label{eq:psi+}
 \psi: Z(H^n)_{2n+j} \to Z(H^{n+1})_{2n+2+j}, \ \ \psi(X_I) = \pm X_I \cdot (X_{2n+2}-X_{2n+1})\end{equation}
and
\begin{equation}
\label{eq:phi+}
 \phi: Z(H^n)_{2n+j} \to Z(H^{n+1})_{2n+4+j}, \ \ \phi(X_I) = \pm X_I \cdot (-X_{2n+1}X_{2n+2}).
 \end{equation}
 \end{proposition}
 
 \begin{proof}
 Consider the saddle cobordism $S$ from the $(n+1, n+1)$-tangle $J$ to the identity $\Id_{n+1, n+1}$:
  \[ 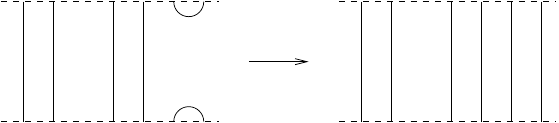
\]
 This induces a cobordism map between the associated $(H^{n+1}, H^{n+1})$-bimodules:
 $$ M=\F(J) \xrightarrow{\phantom{bla}\F(S) \phantom{bla}} \F(\Id_{n+1, n+1}) = H^{n+1}.$$
By restriction, we get a map on Hochschild cohomology
 $$\F(S): \HH^0(M) \to \HH^0(H^{n+1}).$$

By taking circular closures in $S^1 \times S^2$, the cobordism $S$ produces a cobordism in $[0,1] \times S^1 \times S^2$ between the links $F_{n, n} \sqcup U$ and $F_{n+1, n+1}$. Furthermore, by introducing $p$ full twists in place of the $0$-framed unknot, we get the saddle cobordism $Z$ from $T'_{2n, 2np} \sqcup U$ to $T'_{2n+2, (2n+2)p}$ which produces the maps $\psi$ and $\phi$. 

Equation \eqref{eq:introducetwists} relates the Khovanov homology of a link $L \subset S^1 \times S^2$ to its counterpart $L(p) \subset S^3$ obtained by introducing $p>0$ full twists.  We get an  isomorphism in homological degree $i=0$ provided that $L$ has no positive crossings. By the naturality of this isomorphism, there is a commutative diagram
$$\xymatrix{
\KhR^{0, *}_2( T'_{2n, 2np} \sqcup U) \ar[d] \ar[rr]^{\KhR^{0, *}_2(Z)} & & \KhR^{0, *}_2( T'_{2n+2, (2n+2)p}) \ar[d] \\Z(H^n) \otimes \A \ar[r]^{\cong} &   \HH^0(M) \ar[r]^{\F(S)} & \HH^0(H^{n+1}).
 }$$
 The first isomorphism in the bottom row is \eqref{eq:Xi}, which is $\pm \Theta$ according to Proposition~\ref{prop:HM}. Therefore,
 $$ \psi(x) = \pm \F(S)(\Theta(x \otimes 1)), \ \ \ \phi(x) =   \pm \F(S)(\Theta(x \otimes X).$$
 The maps $\psi$ and $\phi$ preserve the $R^{\otimes 2n}$-module action, so to describe them it suffices to evaluate them on $1$. We have
 \begin{align*}\psi(1) &= \pm \F(S)(\Theta(1 \otimes 1)) \\
 &=  \pm \F(S)(1 \otimes \Delta(1) + 1_{\to})\\
 &= \pm \F(S) \Bigl( \sum_{a \in \Cr_n} ({}_{\ta}1_{\ta}) \otimes (1\otimes X + X \otimes 1) +\sum_{a \in \Cr_n} \sum_{e \in \Out(a)} 1_{\to}^e \Bigr) \\
 &= \pm \Bigl(   \sum_{a \in \Cr_n} 2({}_{\ta}(X_{2n+2})_{\ta})+ \sum_{a \in \Cr_n} \sum_{e \in \Out(a)}\bigl(  {}_{a_{\to}^e}(X_{2n+2})_{a_{\to}^e}- {}_{a_{\to}^e}(X_{2n+1})_{a_{\to}^e}\bigr)\Bigr)\\
 & = \pm (X_{2n+2} -X_{2n+1}),
 \end{align*}
 where in the last equation we used the fact that $X_{2n+1} = - X_{2n+2}$ on summands of $H^{n+1}$ of the form $\Hnp{\ta}{\ta}.$ This proves \eqref{eq:psi+}.
 
Moreover,
 \begin{align*}\phi(1) &= \pm \F(S)(\Theta(1 \otimes X)) \\
 &=  \pm \F(S)(1 \otimes \Delta(X) + 1_{\to}\cdot X )\\
 &= \pm \F(S) \Bigl( \sum_{a \in \Cr_n} ({}_{\ta}1_{\ta}) \otimes X\otimes X +\sum_{a \in \Cr_n} \sum_{e \in \Out(a)} 1_{\to}^e \cdot X_{2n+2} \Bigr) \\
 &= \pm \Bigl(0 + \sum_{a \in \Cr_n} \sum_{e \in \Out(a)}    \bigl(  {}_{a_{\to}^e}(-X_{2n+1}X_{2n+2})_{a_{\to}^e}\bigr) \Bigr)\\
 & = \mp (X_{2n+1}X_{2n+2}).
 \end{align*}
This proves \eqref{eq:phi+}.
  \end{proof}

 \begin{proposition}
\label{prop:minus}
Let $p < 0$. Under the identification \eqref{eq:KhRp-}, and using the description of $Z(H^n)$ from Theorem~\ref{thm:ZHn}, the maps $\psi$ and $\phi$ from \eqref{eq:psi}, \eqref{eq:phi} are 
$$ \psi: Z(H^n)^\vee_{2n-j} \to Z(H^{n+1})^\vee_{2n+2-j}$$
given by
\begin{equation}
\label{eq:psi-}
 \psi(f)(X_I) = \pm\begin{cases} 
 f(X_J) & \text{if } I = J \cup \{2n+2\}, \ J \subseteq \{1, \dots, 2n\} \\
- f(X_J) & \text{if } I = J \cup \{2n+1\}, \ J \subseteq \{1, \dots, 2n\} \\
 0 & \text{otherwise,}
 \end{cases}
 \end{equation}
and
$$ \phi: Z(H^n)^\vee_{2n-j} \to Z(H^{n+1})^\vee_{2n-j}$$
given by
\begin{equation}
\label{eq:phi-}
 \phi(f)(X_I) =\pm \begin{cases}
 f(X_I) & \text{if } I \subseteq \{1, \dots, 2n\} \\
 0 & \text{otherwise.}
 \end{cases}
 \end{equation}
 \end{proposition}
 
  \begin{proof}
 This is similar to Proposition~\ref{prop:plus}, except we should consider the reverse saddle cobordism $S^r$ from the identity tangle $\Id_{n+1,  n+1} $ to $J$. This gives a map on Hochschild cohomology
 $$ \F(S^r): \HH^0(H^{n+1}) \to \HH^0(M).$$
We get a commutative diagram
 $$\xymatrix{
\bigl(  \KhR^{0, *}_2( T'_{2n+2, (2n+2)p})\bigr)^\vee \ar[d] \ar[rr]^{(\KhR^{0, *}_2(Z))^\vee} & &  \bigl(\KhR^{0, *}_2( T'_{2n, 2np} \sqcup U)\bigr)^\vee \ar[d] \\\HH^0(H^{n+1}) \ar[r]^{\F(S^r)} &   \HH^0(M) \ar[r]^{\cong} & Z(H^n) \otimes \A
 }$$
 where the isomorphism in the last arrow at the bottom is $\pm \Theta^{-1}$. To compute $\Theta^{-1} \circ \F(S^r)$, observe that this preserves the $R^{\otimes 2n}$-module structure, and therefore it suffices to evaluate it on $1$, $X_{2n+1}$, $X_{2n+2}$ and $X_{2n+1}X_{2n+2}$. A straightforward calculation gives
\begin{align*}
&(\Theta^{-1} \circ \F(S^r))(1) = 1 \otimes 1,\\
& (\Theta^{-1} \circ \F(S^r))(X_{2n+2}) =  (\Theta^{-1} \circ \F(S^r))(-X_{2n+1}) = 1 \otimes X,\\
&  (\Theta^{-1} \circ \F(S^r))(X_{2n+1}X_{2n+2}) = 0.
\end{align*}
This describes (up to sign) the dual of the map $\KhR^{0, *}_2(Z)$. By taking duals, we obtain the desired description of the maps $\psi$ and $\phi$.
 \end{proof}
 
 \begin{remark}
In Propositions~\ref{prop:plus} and \ref{prop:minus} we only specified the maps $\phi$ and $\psi$ up to a sign. We conjecture that all symbols $\pm$ should be $+$, and $\mp$ should be $-$.
 \end{remark}

One last ingredient in the definition of cabled Khovanov-Rozansky homology is the braid group action. In \cite[Section 5]{KhSpringer}, Khovanov proves that the braid group action on $H^n$ induces an action of the symmetric group $S_{2n}$ on the center $Z(H^n)$, which permutes the variables $X_i$. In our case, we are interested in the subgroup $B_{n, n}$, which acts on $Z(H^n)$ via the product $S_n \times S_n$. The first factor permutes the odd variables $X_{2i+1}$, and the second the even variables $X_{2i}$. 

\subsection{Proof of Proposition~\ref{unknotp}}\label{sec:unknotpproof}
From the definition of cabled Khovanov-Rozansky homology, we have
 \[
\cKhR_{2,0}^{0, j}(U, p) =  \Bigl( \bigoplus\limits_{n\in\N} \KhR_2^{0, 2n+j}(T'_{2n, 2np}) \Bigr)/ \sim
\]
where we divide by the linear and transitive closure of the relations of the form
\begin{equation}
\beta_i(b)v \sim v, \ \ \psi(v) \sim 0,\ \ \phi(v) \sim v.
\end{equation}

In the case $p > 0$, using \eqref{eq:KhRp+} we get
\[
\cKhR_{2,0}^{0, j}(U, p) \cong  \Bigl( \bigoplus\limits_{n\in\N} Z(H^n)_{4n+j} \Bigr)/ \sim
\]
Since we  divide by the relations $\psi(v) \sim 0$, and $\psi$ is given in \eqref{eq:psi+} by multiplication with $\pm (X_{2n+2} - X_{2n+1})$, we find that the variables $X_{2n+1}$ and $X_{2n+2}$ are identified in  the quotient. Using \eqref{eq:phi+}, we get that, up to a sign, $\phi$ is given by multiplication with 
$$ X_{2n+1}X_{2n+2} =  X_{2n+1}^2 = 0.$$
Therefore, after dividing by the relations $\phi(v) \sim v$, everything collapses to zero:
$$ \cKhR_{2,0}^{0, j}(U, p)=0.$$

Let us now consider the case $p < 0$. Using \eqref{eq:KhRp-}, we get
\begin{equation}
\label{eq:2-}
\cKhR_{2,0}^{0, j}(U, p) =  \Bigl( \bigoplus\limits_{n\in\N} Z(H^n)^\vee_{-j} \Bigr)/ \sim
\end{equation}
It follows that $\cKhR_{2,0}^{0, *}(U, p)$ is supported in quantum gradings of the form $j=-2k$ for $k \geq 0$.

We start by looking at the quantum grading $j=0$. From Theorem~\ref{thm:ZHn} we see that each $Z(H^n)_0$ is a copy of $\Z$ (generated by $1$), and hence the same is true for $Z(H^n)^\vee_0$. In the equivalence relation all relations of the form $\psi(v) \sim 0$ are trivial because the targets of the maps $\psi$ are in  degrees $j < 0$. The braid group action is the identity, and from \eqref{eq:phi-} we see that the maps 
$$ \phi: Z(H^n)^\vee_0 \to Z(H^{n+1})^\vee_0$$
are isomorphisms. Hence, the relations $\phi(v) \sim v$ identify together all the different $Z(H^n)_0 \cong \Z$, and we have
$$\cKhR_{2, 0}^{0, 0}(U,p) \cong \Z.$$

Next, we look at quantum gradings $j = -2k$ with $k >0$. Using the notation $X_I^\vee$ from Section~\ref{sec:dual}, the formula \eqref{eq:psi-} for the map $\psi$ can be re-written as
$$ \psi (f) = \pm f \cdot (X_{2n+2}^\vee - X_{2n+1}^\vee).$$
Consider the elements $f_\mm \in Z(H^n)_{2k}^\vee$, where $\mm$ is a partial matching of $\{1, \dots, 2n\}$; cf. Lemma~\ref{lem:match}. We have
$$ \psi(f_{\mm}) = \pm f_{\mm \ \cup \{(2n+2, 2n+1)\}}.$$
 Therefore, after dividing by the relations $\psi(v) \sim 0$, all elements of the form $f_{\mm'}$ are set to zero, where $\mm'$ is a partial matching of $\{1, \dots, 2n+2\}$ that contains the last pair $(2n+2, 2n+1)$.
 
 If $\mm$ is a nonempty balanced partial matching of $\{1, 2, \dots, 2n+2\}$, the action of a suitable element in the braid group $B_{n+1, n+1}$ (factoring through $S_{n+1} \times S_{n+1}$) will take $f_{\mm}$ to $\pm f_{\mm'}$ where $\mm'$ is a balanced matching containing the pair $(2n+2,  2n+1)$.
 It follows that all $f_{\mm}$ are set to zero, for nonempty balanced partial matchings $\mm$. Proposition~\ref{prop:balancedmatch} says that these elements $f_{\mm}$ generate $Z(H^{n+1})_{2k}^\vee$, and therefore the whole group collapses to zero after we divide by the equivalence relation.

This concludes the proof of Proposition~\ref{unknotp} and hence of Theorem~\ref{thm:Dp}.

\section{Connected Sums}
\label{sec:ConnSum}
In this section, we prove Theorem~\ref{connectsum}. We will work with coefficients in a field $\k$. We write $\KhRN(L; \k)$ for the $\gl_N$ Khovanov-Rozansky homology of the framed link $L$ with coefficients in $\k$. We write $\Sz(W; L; \k)$ for the skein lasagna module obtained using $\KhRN(L; \k)$ instead of $\KhRN(L)$.
\begin{remark}
If char$(\k) =0$, then $\KhRN(L; \k) \cong \KhRN(L) \otimes_{\Z} \k$ and $ \Sz(W; L; \k) \cong \Sz(W; L) \otimes_{\Z} \k.$ In general, this is not true, because of the presence of Tor terms in the universal coefficients theorem.
\end{remark}

Recall that, over a field, we have the following tensor product formula for the Khovanov-Rozansky homology of a split union:
\begin{equation}
\label{eq:splittensor}
\KhRN(L_1\sqcup L_2; \k)\cong\KhRN(L_1;\k)\otimes_{\k}\KhRN(L_2;\k)
\end{equation}
We will be interested in the case when the two links are mirror to each other. In that case, the cylinder $L\times I$ is a cobordism from the empty link to $L\sqcup\overline{L}$. This provides a canonical element
$$   \mathfrak{B}: =\KhR_N(C)(1)  \in \KhRN(L \sqcup \overline{L};\k))\cong \KhRN(L; \k)\otimes_{\k} \KhRN(\overline{L};\k).$$
In order to decompose $\mathfrak{B}$ into a sum of simple tensors, we pick a basis $\{u_i\}$ for $\KhRN(L; \k)$ and write
$$   \mathfrak{B} =\sum_i u_i\otimes w_i \in \KhRN(L; \k)\otimes_{\k} \KhRN(\overline{L};\k)$$
for some $w_i\in \KhRN(\overline{L};\k)$. In particular, the right-hand side of the above is independent of the choice of basis $\{u_i\}$.

We will need a lemma for cutting necks of surfaces in lasagna fillings.
\begin{lemma}\label{neckcut}
Let $F$ be a lasagna filling of $W$ with boundary $L$ and surface $\Sigma$. Let $B$ be a 4-ball in the interior of $W$ disjoint from the input balls and intersecting $\Sigma$. Suppose the intersection is of the form  $\Sigma\cap B= \Sigma\cap (B^3\times I)= K\times I$ for some link $K\subset B^3$. Then, we can decompose $\Sigma=\Sigma^\circ\cup (K\times I)$. Let $\{u_i\}$ be a basis of $\KhRN(K; \k)$ so that $\mathfrak{B} =\sum_i u_i\otimes w_i$. Define $\widetilde{F}(u_i)$ to be the lasagna filling specified by the same data as $F$, except we replace the neck $K\times I$ with two input balls: one decorated with the link $K$ and labelled by $u_i$ and the other decorated with $\overline{K}$ and labelled by $w_i$ (See Figure~\ref{fig:neckcut}). Then, $$[F]=\sum\limits_i[\widetilde{F}(u_i)] \in \Sz(W;K; \k).$$
In particular, the right-hand side of this equation is independent of the choice of basis $\{u_i\}$.
\end{lemma}

\begin {figure}
\begin {center}
\def\svgwidth{10cm}
\begingroup%
  \makeatletter%
  \providecommand\color[2][]{%
    \errmessage{(Inkscape) Color is used for the text in Inkscape, but the package 'color.sty' is not loaded}%
    \renewcommand\color[2][]{}%
  }%
  \providecommand\transparent[1]{%
    \errmessage{(Inkscape) Transparency is used (non-zero) for the text in Inkscape, but the package 'transparent.sty' is not loaded}%
    \renewcommand\transparent[1]{}%
  }%
  \providecommand\rotatebox[2]{#2}%
  \newcommand*\fsize{\dimexpr\f@size pt\relax}%
  \newcommand*\lineheight[1]{\fontsize{\fsize}{#1\fsize}\selectfont}%
  \ifx\svgwidth\undefined%
    \setlength{\unitlength}{283.46456693bp}%
    \ifx\svgscale\undefined%
      \relax%
    \else%
      \setlength{\unitlength}{\unitlength * \real{\svgscale}}%
    \fi%
  \else%
    \setlength{\unitlength}{\svgwidth}%
  \fi%
  \global\let\svgwidth\undefined%
  \global\let\svgscale\undefined%
  \makeatother%
  \begin{picture}(1,0.2145974)%
    \lineheight{1}%
    \setlength\tabcolsep{0pt}%
    \put(0,0){\includegraphics[width=\unitlength]{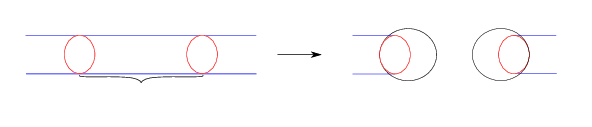}}%
    \put(0.1936,0.04399645){\makebox(0,0)[lt]{\lineheight{1.25}\smash{\begin{tabular}[t]{l}$L\times I$\end{tabular}}}}%
    \put(0.90171115,0.1117469){\makebox(0,0)[lt]{\lineheight{1.25}\smash{\begin{tabular}[t]{l}$\overline{L}$\end{tabular}}}}%
    \put(0.6089246,0.11005465){\makebox(0,0)[lt]{\lineheight{1.25}\smash{\begin{tabular}[t]{l}$L$\end{tabular}}}}%
  \end{picture}%
\endgroup%

\caption {The neck-cutting lemma.}
\label{fig:neckcut}
\end {center}
\end {figure}
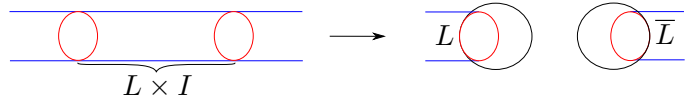

\begin{proof}
First observe that we can always add an input ball decorated by the empty set and labelled by $1\in\KhRN(\emptyset)$ to any filling without affecting its class in $\Sz$. Add such an input ball near the neck $K\times I$. Enclose this new input ball together with the neck inside of a larger ball, as in Figure ~\ref{neckcutlemmafigs}. We can thereby view the cylindrical neck $K\times I$ as a cobordism from the empty set to $K\sqcup \overline{K}$. The image of $1\in\KhRN(\emptyset)$ under this cobordism map is $\mathfrak{B}$. Evaluating this cobordism gives a sum of fillings with the new ball decorated with $K\sqcup \overline{K}$ and labelled by the $u_i\otimes w_i$. The claim follows by splitting this input ball into two balls, one with $K$ and the other with $\overline{K}$.
\end{proof}

\begin {figure}
\begin {center}
\def\svgwidth{15cm}
\begingroup%
  \makeatletter%
  \providecommand\color[2][]{%
    \errmessage{(Inkscape) Color is used for the text in Inkscape, but the package 'color.sty' is not loaded}%
    \renewcommand\color[2][]{}%
  }%
  \providecommand\transparent[1]{%
    \errmessage{(Inkscape) Transparency is used (non-zero) for the text in Inkscape, but the package 'transparent.sty' is not loaded}%
    \renewcommand\transparent[1]{}%
  }%
  \providecommand\rotatebox[2]{#2}%
  \newcommand*\fsize{\dimexpr\f@size pt\relax}%
  \newcommand*\lineheight[1]{\fontsize{\fsize}{#1\fsize}\selectfont}%
  \ifx\svgwidth\undefined%
    \setlength{\unitlength}{425.19685039bp}%
    \ifx\svgscale\undefined%
      \relax%
    \else%
      \setlength{\unitlength}{\unitlength * \real{\svgscale}}%
    \fi%
  \else%
    \setlength{\unitlength}{\svgwidth}%
  \fi%
  \global\let\svgwidth\undefined%
  \global\let\svgscale\undefined%
  \makeatother%
  \begin{picture}(1,0.39921722)%
    \lineheight{1}%
    \setlength\tabcolsep{0pt}%
    \put(0,0){\includegraphics[width=\unitlength]{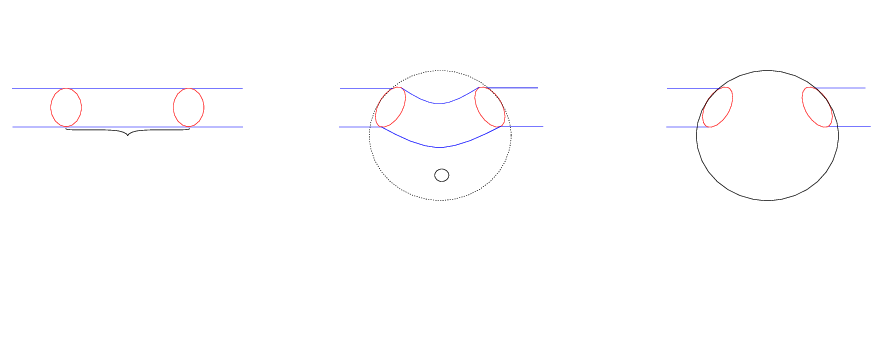}}%
    \put(0.1067922,0.22201801){\makebox(0,0)[lt]{\lineheight{1.25}\smash{\begin{tabular}[t]{l}$L\times I$\end{tabular}}}}%
    \put(0.43617658,0.17425026){\makebox(0,0)[lt]{\lineheight{1.25}\smash{\begin{tabular}[t]{l}$1\in\operatorname{KhR}_N (\emptyset)$\end{tabular}}}}%
    \put(0.77872763,0.32498985){\makebox(0,0)[lt]{\lineheight{1.25}\smash{\begin{tabular}[t]{l}$\mathfrak{B}\in\operatorname{KhR}_N (L\sqcup\overline{L})$\end{tabular}}}}%
    \put(0.04476355,0.09984192){\makebox(0,0)[lt]{\lineheight{1.25}\smash{\begin{tabular}[t]{l}The neck $L\times I$ \\ inside of the surface $\Sigma$\end{tabular}}}}%
    \put(0.38786664,0.10290075){\makebox(0,0)[lt]{\lineheight{1.25}\smash{\begin{tabular}[t]{l}Add a new input ball \\ decorated with $\emptyset$ \\ and labelled by $1$.\end{tabular}}}}%
    \put(0.70829982,0.10550255){\makebox(0,0)[lt]{\lineheight{1.25}\smash{\begin{tabular}[t]{l}Evaluate the cobordism \\ in the enclosed region \\ to obtain an \\ equivalent diagram.\end{tabular}}}}%
  \end{picture}%
\endgroup%

\caption {Proof of the neck-cutting lemma.}
\label{neckcutlemmafigs}
\end {center}
\end {figure}
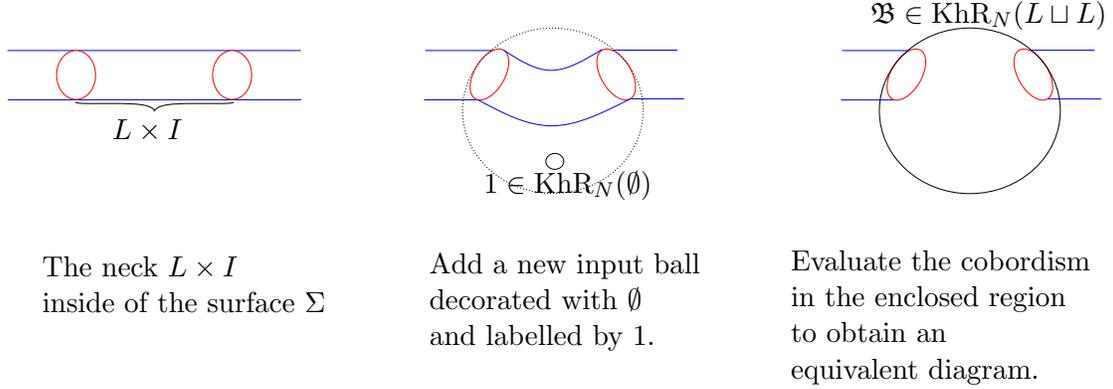

We now provide the proof of the tensor product formula for boundary connected sums.
\begin{proof}[Proof of Theorem ~\ref{connectsum}]
We define the isomorphism $$\Psi: \Sz(W_1; L_1; \k)\otimes \Sz(W_2; L_2)\to \Sz(W_1\natural W_2; L_1\cup L_2; \k)$$ on simple tensors by setting $\Psi([F_1]\otimes [F_2])$ to be the lasagna filling represented by $F_1\cup F_2$.

We define an inverse to $\Psi$ as follows. The boundary connected sum is obtained from $W_1$ and $W_2$ by identifying $3$-dimensional balls $B_1 \subset \del W_1$ and $B_2 \subset \del W_2$; we write $B$ for $B_1 = B_2$ as a subset of $W_1 \natural W_2$. Let $F$ be a lasagna filling of $W_1 \natural W_2$ with boundary $L_1 \cup L_2$ and surface $\Sigma$. After an isotopy, we can arrange that:
\begin{enumerate}[(a)]
\item The input balls for $F$ are disjoint from $B$;
\item The surface $\Sigma$ intersects $B$ transversely in a link $K$.
\end{enumerate}
Decompose $\Sigma=\Sigma_1\cup_K\Sigma_2$ where $\Sigma_i\subset W_i$. We can apply Lemma \ref{neckcut} to cut along $K$ and obtain 
$$[F] = \sum\limits_i[\widetilde{F}(u_i)] \in \Sz(W_1\natural W_2; L_1\cup L_2; \k)$$
where each $\widetilde{F}(u_i)$ is of the form $F^1_i\cup F^2_i$, with fillings $F^j_i$ of $W_j$ with boundary $L_j$, $j=1,2$. Then $\Psi(\sum_i [F^1_i]\otimes [F^2_i])=[F]$, and we set
$$ \Psi^{-1}([F]) = \sum_i [F^1_i]\otimes [F^2_i].$$

We need to make sure that $\Psi^{-1}$ is well-defined. Filling in one of the input balls of $F$ (in either $W_1$ or $W_2$) with another lasagna filling does not change the equivalence classes $ [F^1_i]$ and  $[F^2_i]$, so the value of $\Psi^{-1}$ is unchanged. 

What is left to show is that  $\Psi^{-1}([F])$ does not depend on the choice of isotopy used to ensure the conditions (a) and (b) above. Consider an isotopy that moves the lasagna filling $F=F_{(0)}$ in a family $F_{(t)}, t \in [0,1]$, such that the final filling $F_{(1)}$ also satisfies (a) and (b). 

With regard to (a), we can imagine the input balls of the fillings to be small (i.e., neighborhoods of points). Generically, in a one-parameter family such as $F_{(t)}$, there can be finitely many times $t$ where an input ball passes from one side of $B$ to the other. Moving the input ball to the other side is equivalent to replacing $B$ with an isotopic ball $B'$, such that the region between $B$ and $B'$ is a cylinder $B^3 \times [0,1]$: 
\[ 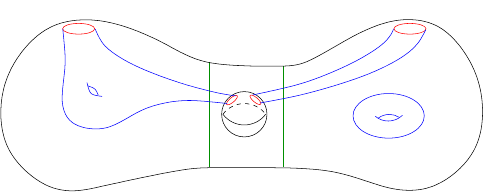
\]
We obtain $\Psi^{-1}([F])$ in one case by cutting the filling $F$ along $B$, and in the other case by cutting it along $B'$. By Lemma~\ref{neckcut}, both of these are equivalent to cutting along both $B$ and $B'$, and therefore equivalent to each other.

To deal with (b), without loss of generality, we can now assume that the input balls do not intersect $B$ throughout the isotopy. Let $\Sigma_{(t)}$ be the surfaces for $F_{(t)}$. The intersections $$J_{(t)} := \Sigma_{(t)} \cap B$$
may fail to be transverse at various points in $(0,1)$, but generically we can assume that the union
$$ J = \bigcup_{t \in [0,1]} \bigl( \{t\} \times J_{(t)} \bigr) \subset [0,1] \times B$$
is a smooth link cobordism between the links $J_{(0)}$ and $J_{(1)}$. Schematically, we draw this as:
\[ 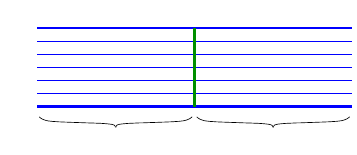
\]
We use the cobordism $J$ to construct a surface with corners as follows. First, thicken the attaching region $B$ of $W_1\natural W_2$ to a neck $[0,1]\times B$. Then, insert a copy of $J$ into this neck. Finally, take the union of $J$ with $\Sigma_{(0)}$ in the $W_1$ factor and $\Sigma_{(1)}$ in the $W_2$ to obtain
$$ \bigl(\Sigma_{(0)}|_{W_1}\bigr) \cup J \cup \bigl(\Sigma_{(1)}|_{W_2}\bigr).$$
By smoothing the corners of this surface, we obtain a new lasagna filling $F_J$. This is isotopic to $F_{(0)}$ by an isotopy supported in $W_2$ and is isotopic to $F_{(1)}$ by an isotopy supported in $W_1$:
\[ 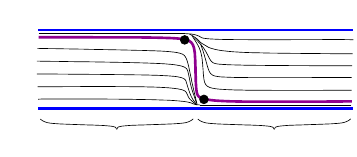
\]
For example, the isotopy between $F_{(0)}$ and $F_J$ is given at time $t$ by smoothing the corners of 
$$ \bigl( \Sigma_{(0)}|_{W_1} \bigr)\cup \bigcup_{s \in [0,t]} \bigl(\{s\} \times J_{(s)}\bigr)  \cup \bigl(\Sigma_{(t)}|_{W_2}\bigr).$$

Applying $\Psi^{-1}$ to $F_{(0)}$ consists in cutting its neck at $\{0\} \times B$, which is equivalent  to cutting the neck of $F_J$ at $\{0\} \times B$ (because they are related by an isotopy supported in $W_2$). Similarly, applying $\Psi^{-1}$ to $F_{(1)}$ is equivalent to cutting the neck of $F_J$ at $\{1\} \times B$. From Lemma~\ref{neckcut} we see that the results of cutting $F_J$ at $\{0\} \times B$ and $\{1\} \times B$ are equivalent, because they are each equivalent to cutting the neck in both places. 

This completes the proof of well-definedness for $\Psi^{-1}$. The fact that $\Psi$ and $\Psi^{-1}$ are inverse to each other is immediate from the construction. 
\end{proof}

We can also deduce the same result for interior connected sums.
\begin{corollary}
\label{cor:InteriorConnSum}
Let $(W_1;L_1)$ and $(W_2;L_2)$ be a pair of 4-manifolds with links in the boundaries. Let $W_1\# W_2$ denote their interior connected sum. Then,
\begin{align*}
\Sz(W_1\# W_2; L_1\cup L_2; \k)\cong \Sz(W_1; L_1; \k)\otimes \Sz(W_2; L_2; \k)
\end{align*}
\end{corollary}
\begin{proof}
By Proposition \ref{3and4handles}, we can add and remove small 4-balls without affecting $\Sz$. We remove a small 4-ball from each of the $W_i$, then perform the boundary connect sum along 3-balls in the new 3-sphere boundary components. The two boundaries glue together to give a 3-sphere boundary component in the connected sum, which we then fill in with a 4-ball to obtain $W_1\#W_2$. Applications of Theorem~\ref{connectsum} and Proposition \ref{3and4handles} give the result.
\end{proof}

\bibliographystyle{custom}
\bibliography{biblio}

\end{document}